\documentclass[12pt]{amsart}
\usepackage{amsfonts,amssymb,amsmath,amscd,amsthm}
\usepackage{graphicx} 
\usepackage{stmaryrd}
\usepackage[all]{xy}

\usepackage[colorlinks,allcolors=blue]{hyperref} 

\newtheorem{thm}{Theorem}
\newtheorem{lem}[thm]{Lemma}
\newtheorem{prop}[thm]{Proposition}
\newtheorem{claim}[thm]{Claim}
\newtheorem{cor}[thm]{Corollary}

\theoremstyle{definition}
\newtheorem{dfn}[thm]{Definition} 
\newtheorem{ex}[thm]{Example}
\newtheorem{rmk}[thm]{Remark}

\numberwithin{thm}{section}
\numberwithin{equation}{section}

\newcommand{\Hom}{\operatorname{Hom}}
\newcommand{\Ext}{\operatorname{Ext}}
\newcommand{\Spec}{\operatorname{Spec}}

\newcommand{\ob}{\operatorname{ob}}
\newcommand{\Hilb}{\operatorname{Hilb}}
\newcommand{\HF}{\operatorname{HF}}
\newcommand{\red}{\operatorname{red}}
\newcommand{\im}{\operatorname{im}}
\renewcommand{\div}{\operatorname{div}}

\newcommand{\Pic}{\operatorname{Pic}}
\newcommand{\eff}{\operatorname{NE}}
\newcommand{\mori}{\overline{\operatorname{NE}}}
\newcommand{\amp}{\operatorname{Amp}}
\newcommand{\ns}{\operatorname{NS}}
\newcommand{\coker}{\operatorname{coker}}

\newcommand{\car}{\operatorname{char}}

\newcommand{\ov}[1]{\overline{#1}}

\renewcommand{\theenumi}{\arabic{enumi}}

\newcommand{\mapright}[1]{%
\smash{\mathop{%
\hbox to 1cm{\rightarrowfill}}\limits^{#1}}}
\newcommand{\mapleft}[1]{%
\smash{\mathop{%
\hbox to 1cm{\leftarrowfill}}\limits^{#1}}}
\newcommand{\mapdown}[1]{\Big\downarrow
\llap{$\vcenter{\hbox{$\scriptstyle#1\,$}}$ }}
\newcommand{\mapup}[1]{\Big\uparrow
\rlap{$\vcenter{\hbox{$\scriptstyle#1$}}$ }}

\title[Obstructions to deforming curves on a $3$-fold, III]{
Obstructions to deforming curves on a 3-fold, III:\\
Deformations of curves lying on a $K3$ surface
}
\author{Hirokazu Nasu}
\date{}
\address{
Department of Mathematical Sciences,
Tokai University,
4-1-1 Kitakaname, Hiratsuka, 
Kanagawa 259-1292, JAPAN}
\email{nasu@tokai-u.jp}

\begin{document}


\begin{abstract}
We study the deformations of a smooth curve $C$
on a smooth projective $3$-fold $V$, assuming the presence of
a smooth surface $S$ satisfying $C \subset S \subset V$.
Generalizing a result of Mukai and Nasu,
we give a new sufficient condition
for a first order infinitesimal deformation of $C$ in $V$ 
to be primarily obstructed. 
In particular, when $V$ is Fano and $S$ is $K3$, 
we give a sufficient condition for $C$
to be (un)obstructed in $V$,
in terms of $(-2)$-curves and elliptic curves on $S$.
Applying this result, we prove that
the Hilbert scheme $\Hilb^{sc} V_4$
of smooth connected curves
on a smooth quartic $3$-fold $V_4 \subset \mathbb P^4$
contains infinitely many
generically non-reduced irreducible components,
which are variations of Mumford's example for $\Hilb^{sc} \mathbb P^3$.
\end{abstract}

\maketitle


\section{Introduction}

Let $V$ be a smooth projective $3$-fold over 
an algebraically closed field $k$.
This paper is a sequel to the preceding papers
\cite{Mukai-Nasu,Nasu4}, in which 
the embedded deformations of a smooth curve $C$ in $V$
have been studied
under the presence of an intermediate 
smooth surface $S$ satisfying $C \subset S \subset V$.
As is well known, first order (infinitesimal) deformations
$\tilde C \subset V \times \Spec k[t]/(t^2)$ of $C$ in $V$ 
are in one-to-one correspondence with
global sections $\alpha$ of the normal bundle $N_{C/V}$ of $C$ in $V$.
Then the obstruction $\ob(\alpha)$
to lifting $\tilde C$ to a second order deformation 
$\tilde {\tilde C} \subset V \times \Spec k[t]/(t^3)$ of $C$ in $V$
is contained in $H^1(C,N_{C/V})$, 
and $\ob(\alpha)$ is computed as
a cup product $\alpha \cup \alpha$ by the map
$$
H^0(C,N_{C/V})\times H^0(C,N_{C/V}) \overset{\cup}\longrightarrow H^1(C,N_{C/V})
$$
(cf.~Theorem~\ref{thm:original cup product}).
It is generally difficult to compute $\ob(\alpha)$ directly.
Mukai and Nasu \cite{Mukai-Nasu} introduced {\em the exterior components}
of $\alpha$ and $\ob(\alpha)$, 
which are defined as
the images of $\alpha$ and $\ob(\alpha)$
in $H^i(C,N_{S/V}\big{\vert}_C)$ ($i=0,1$) 
by the natural projection
$\pi_{C/S}: N_{C/V} \rightarrow N_{S/V}\big{\vert}_C$,
and denoted by
$\pi_{C/S}(\alpha)$ and $\ob_{S}(\alpha)$, respectively
(cf.~\S\ref{subsec:exterior}).
They gave a sufficient condition for $\ob_S(\alpha)$ to be nonzero,
which implies the non-liftability of $\tilde C$ to any $\tilde {\tilde C}$.

In this paper, we generalize their result and give a weaker
condition for $\ob_S(\alpha)\ne 0$.
Let $\tilde C$ or $\alpha \in H^0(C,N_{C/V})$ be a first order 
deformation of $C$ in $V$, and 
$\pi_{C/S}(\alpha) \in H^0(C,N_{S/V}\big{\vert}_C)$ 
the exterior component of $\alpha$.
Suppose that the image of $\pi_{C/S}(\alpha)$
in $H^0(C,N_{S/V}(mE)\big{\vert}_C)$
lifts to a section
$$
\beta \in H^0(S,N_{S/V}(mE))
\qquad
$$
for an integer $m \ge 1$ and an effective Cartier divisor $E$ on $S$.
In other words, we have
$\pi_{C/S}(\alpha)= \beta\big{\vert}_C$ in $H^0(C,N_{S/V}(mE)\big{\vert}_C)$.
Here $\beta$ is called
an {\em infinitesimal deformations with a pole} (along $E$)
of $S$ in $V$ (cf.~\S\ref{subsec:with pole}).
The following is a generalization of \cite[Theorem 2.2]{Mukai-Nasu},
in which, it was assumed that $m=1$, 
$E$ is smooth and irreducible with
negative self-intersection number $E^2<0$ on $S$,
and furthermore, $E$ was assumed to be a $(-1)$-curve on $S$
(i.e.~$E \simeq \mathbb P^1$ and $E^2=-1$) in its application.

\begin{thm}\label{thm:main1}
  Let $\tilde C$, $\alpha$, $\beta$ be as above.
  Suppose that the natural map
  $H^1(S,\mathcal O_S(kE))\rightarrow H^1(S,\mathcal O_S((k+1)E))$
  is injective for every integer $k \ge 1$.
  If the following conditions are satisfied,
  then the exterior component 
  $\ob_S(\alpha)$ of $\ob(\alpha)$ is nonzero:
  \begin{enumerate}
    \renewcommand{\theenumi}{{\alph{enumi}}}
    \item the restriction map 
    $
    H^0(S,\Delta) \overset{\vert_E}\longrightarrow 
    H^0(E,\Delta\big{\vert}_E)
    $
    is surjective for $\Delta:=C+K_V\big{\vert}_S-2mE$, a divisor on $S$,
    and
    \item we have
    $$
    m\partial_E(\beta\big{\vert}_E) \cup
    \beta\big{\vert}_E \ne 0 
    \qquad
    \mbox{in}
    \qquad
    H^1(E,N_{S/V}((2m+1)E-C)\big{\vert}_E),
    $$
    where $\beta\big{\vert}_E \in H^0(E,N_{S/V}(mE)\big{\vert}_E)$
    is the principal part of $\beta$ along the pole $E$,
    and $\partial_E$ is the coboundary map of the exact sequence
    \begin{equation}
      \label{ses:normal bundle of E}
      [0 \longrightarrow \underbrace{N_{E/S}}_{\simeq \mathcal O_E(E)}
	\longrightarrow N_{E/V}
	\overset{\pi_{E/S}}\longrightarrow N_{S/V}\big{\vert}_E
	\longrightarrow 0]
      \otimes_{\mathcal O_E} \mathcal O_E(mE).   
    \end{equation}
  \end{enumerate}
\end{thm}

The relation between $\alpha$ and $\beta\big{\vert}_E$ is explained
with Figure~\ref{fig:relation} in \S\ref{sect:obstruction}.

\medskip

Given a projective scheme $V$,
let $\Hilb^{sc} V$ denote the Hilbert scheme
of smooth connected curves in $V$.
Mumford~\cite{Mumford} first proved that 
$\Hilb^{sc} \mathbb P^3$ contains
a generically non-reduced (irreducible) component.
Later, many examples of such non-reduced 
components of $\Hilb^{sc} \mathbb P^3$
were found in \cite{Kleppe87,Ellia87,Gruson-Peskine,Floystad,Nasu1}, etc.
More recently, Mumford's example was generalized in \cite{Mukai-Nasu}
and it was proved that for many uniruled $3$-folds $V$,
$\Hilb^{sc} V$ contains infinitely many generically non-reduced components.
(See \cite{Vakil} for a different generalization.)
In the construction of the components, $(-1)$-curves $E \subset V$ 
on a surface $S \subset V$ play a very important role.
In this paper, as an application we study the deformations of curves $C$
on a smooth Fano $3$-fold $V$ when
$C$ is contained in a smooth $K3$ surface $S\subset V$.
On a $K3$ surface, $(-2)$-curves $E$
(i.e.~$E \simeq \mathbb P^1$ and $E^2=-2$)
and elliptic curves $F$ (then $F^2=0$)
play a role very similar to that of a $(-1)$-curve.
If we have $m=1$ in the exact sequence ~\eqref{ses:normal bundle of E},
then the sheaf homomorphism
$\pi_{E/S}\otimes_{\mathcal O_E} \mathcal O_E(E)$ 
tensored with $\mathcal O_E(E)$ induces a map 
(called the ``$\pi$-map'' for $(E,S)$)
\begin{equation}
  \label{map:pi-map}
  \pi_{E/S}(E): 
  H^0(E,N_{E/V}(E))
  \longrightarrow
  H^0(E,N_{S/V}(E)\big{\vert}_E)
\end{equation}
on the cohomology groups.
From now on, we assume that $\car(k)=0$.

\begin{thm}\label{thm:main2}
 Let $V$ be a smooth Fano $3$-fold, 
 $S \subset V$ a smooth $K3$ surface, and
 $C \subset S$ a smooth connected curve,
 and put $D:=C+K_V\big{\vert}_S$ a divisor on $S$.
 Suppose that there exists a first order deformation $\tilde S$ of $S$
 which does not contain any first order deformations $\tilde C$ of $C$.
  \begin{enumerate}
    \item If $D \ge 0$ and
    there exist no $(-2)$-curves and no elliptic curves on $S$,
    or more generally, if $H^1(S,D)=0$, 
    then $\Hilb^{sc} V$ is nonsingular at $[C]$.
    \item If $D \ge 0$, $D^2 \ge 0$
    and there exists a $(-2)$-curve $E$ on $S$
    such that $E.D=-2$ and $H^1(S,D-3E)=0$,
    then we have $h^1(S,D)=1$.
    If moreover, $\pi_{E/S}(E)$ is not surjective, then
    $\Hilb^{sc} V$ is singular at $[C]$.
    \item If there exists 
    an elliptic curve $F$
    on $S$ such that $D \sim mF$ for an integer $m \ge 2$,
    then we have $h^1(S,D)=m-1$.
    If moreover, $\pi_{F/S}(F)$ is not surjective, 
    then $\Hilb^{sc} V$ is singular at $[C]$.
  \end{enumerate}
\end{thm}

Note that $H^1(S,D)\simeq H^1(S,N_{S/V}(-C))^{\vee}$
(cf.~\eqref{isom:cohomology of abnormality}).
If $H^1(S,D)=0$, then $C$ is unobstructed in $V$.
By using Theorem~\ref{thm:main1},
we partially prove that $C$ is obstructed in $V$ if $H^1(S,D)\ne 0$.
Under the assumption of Theorem~\ref{thm:main2}, 
the Hilbert-flag scheme $\HF V$ of $V$ (cf.~\S\ref{subsec:flag schemes})
is nonsingular at $(C,S)$
(cf.~Lemma~\ref{lem:flag of k3fano}).
Therefore $(C,S)$ belongs to a unique irreducible component
$\mathcal W_{C,S}$ of $\HF^{sc} V$.
The image $W_{C,S}$ of $\mathcal W_{C,S}$ in $\Hilb^{sc} V$ is called
the {\em $S$-maximal family} of curves containing $C$
(cf.~Definition~\ref{dfn:S-maximal}).
Then $W_{C,S}$ is of codimension $h^1(S,D)$
in the tangent space $H^0(C,N_{C/V})$ of $\Hilb^{sc} V$ at $[C]$
(cf.~\eqref{ses:flag to hilb}).

\begin{cor}
 \label{cor:main3}
 In (1),(2),(3) of Theorem~\ref{thm:main2},
 we have furthermore that
 \begin{enumerate}
  \renewcommand{\theenumi}{{\alph{enumi}}}
  \item If $h^1(S,D)\le 1$, then
	$W_{C,S}$ is an irreducible component of $(\Hilb^{sc} V)_{\red}$.
  \item $\Hilb^{sc} V$ is generically smooth along $W_{C,S}$
	if $h^1(S,D)=0$, and generically non-reduced along $W_{C,S}$
	if $h^1(S,D)=1$.
  \item If $H^0(S,-D)=0$, then
	$\dim_{[C]} \Hilb^{sc} V
	=(-K_V\big{\vert}_S)^2/2+g(C)+1$,
	where $g(C)$ is the genus of $C$.
 \end{enumerate}
\end{cor}

The following is a simplification or a variation of Mumford's example.
(See Examples \ref{ex:non-reduced V_4} and \ref{ex:non-reduced P^3}
  for more examples.)

\begin{ex}[$\car k=0$]
  \label{ex:non-reduced}
  In the following examples,
  the closure $\overline W$ of $W$ is an irreducible component of 
  $(\Hilb^{sc} V)_{\red}$ and 
  $\Hilb^{sc} V$ is generically non-reduced along $W$.
  We have $h^0(C,N_{C/V})=\dim W+1$ at the generic point $[C]$ of $W$.
  \begin{enumerate}
    \item Let $V$ be a smooth quartic $3$-fold $V_4 \subset \mathbb P^4$,
    $E$ a smooth conic on $V$ 
    with trivial normal bundle $N_{E/V} \simeq \mathcal O_E^{2}$,
    $S$ a smooth hyperplane section of $V$ containing $E$ and
    such that $\Pic S=\mathbb Z\mathbf h\oplus \mathbb Z E$,
    where $\mathbf h\sim \mathcal O_S(1)$.
    Then a general member $C$ of the complete linear system 
    $|2\mathbf h+2E|$ on $S$ is
    a smooth connected curve of degree $12$ and genus $13$.
    Such curves $C$ are parametrised by a locally closed
    irreducible subset $W$ of $\Hilb^{sc} V$ of dimension $16$.
    \item Let $V=\mathbb P^3$ 
    and let $F$ be a smooth plane cubic (elliptic) curve,
    $S$ a smooth quartic surface containing $F$.
    Then a general member $C$ of 
    $|4\mathbf h+2F|$ ($\mathbf h \sim \mathcal O_S(1)$)
    is a smooth connected curve
    of degree $22$ and genus $57$ on $S$.
    Such curves $C$ are parametrised by a locally closed irreducible
    subset $W$ of $\Hilb^{sc} \mathbb P^3$ of dimension $90$.
  \end{enumerate}
\end{ex}

The organization of this paper is as follows.
The proof of Theorem~\ref{thm:main1} heavily depends on the 
analysis of the singularity of {\em polar $d$-maps}.
Given a projective scheme $V$ and its hypersurface $S \subset V$,
there exists a so-called ``Hilbert-Picard'' morphism
$\psi_{S}: \Hilb^{cd} V \rightarrow \Pic S$
from the Hilbert scheme of effective Cartier divisors on $V$
to the Picard scheme of $S$, 
sending a hypersurface $S' \subset V$ to 
the invertible sheaf $\mathcal O_V(S')\big{\vert}_S$ on $S$
(cf.~\S\ref{subsec:hypersurface}).
The tangent map $d_S: H^0(S,N_{S/V}) \rightarrow H^1(S,\mathcal O_S)$ 
of $\psi_S$ at $[S]$ is called the {\em $d$-map} for $S\subset V$.
In \S\ref{subsec:with pole}, we show that this map is extended
into a version
$d_{S^\circ}: H^0(S,N_{S/V}(mE)) \rightarrow H^1(S,\mathcal O_S((m+1)E))$
with a pole along a divisor $E\ge 0$ on $S$.
We also prove that for any $\beta \in H^0(S,N_{S/V}(mE))$
the restriction $d_{S^\circ}(\beta)\big{\vert}_E$ to $E$
of the image $d_{S^\circ}(\beta)$ 
coincides with the coboundary image
$\partial_E(\beta\big{\vert}_E)$ of \eqref{ses:normal bundle of E}
up to constant
(cf.~Proposition~\ref{prop:key}).
In \S\ref{sect:obstruction},
applying this result to a $3$-fold $V$, 
we prove Theorem~\ref{thm:main1}.
In \S~\ref{sect:k3}, we prove
Theorem~\ref{thm:main2} and Corollary~\ref{cor:main3}
as an application of Theorem~\ref{thm:main1}.
In \S\ref{subsec:mori cone of quartics}, 
we study the Mori cone $\mori(S_4)$ of
a smooth quartic surface $S_4$ of Picard number $2$
(cf.~Lemmas~\ref{lem:mori cone1}, \ref{lem:mori cone2}).
Applying the results on Mori cones, in \S\ref{subsec:hilb},
for curves $C$ lying on $S_4$,
we study the deformations of $C$ in $\mathbb P^3$
and $C$ in a smooth quartic $3$-fold $V_4 \subset \mathbb P^4$.
In particular,
we give a sufficient condition for $W_{C,{S_4}}$ to be
a generically non-reduced (or generically smooth) component
of $\Hilb^{sc} \mathbb P^3$ and $\Hilb^{sc} V_4$
(cf.~Theorems~\ref{thm:Hilb with rational curve},
\ref{thm:Hilb with elliptic curve} and
\ref{thm:Hilb without rational nor elliptic curve}).

\section{Preliminaries}
\label{sect:preliminarlies}

We start by recalling some basic facts on the deformation theory of 
closed subschemes, as well as setting up Notations. 
We work over an algebraically closed field 
$k$ of characteristic $p \ge 0$. 
Let $V\subset \mathbb P^n$ be a closed subscheme of $\mathbb P^n$ 
with the embedding invertible sheaf $\mathcal O_V(1)$ on $V$,
and $X \subset V$ a closed subscheme of $V$ with 
the Hilbert polynomial $P_X=\chi(\mathcal O_X(n))$.
Then as is well known,
the {\em Hilbert scheme} $\Hilb_{P} V$ of $V$ 
parametrises all closed subscheme
$X'$ of $V$ with $P_{X'}=P_X$ (cf.~\cite{Grothendieck}).
We denote by $\Hilb V$ the (full) Hilbert scheme 
$\bigsqcup_{P} \Hilb_P V$ of $V$.
Let $\mathcal I_X$ and 
$N_{X/V}=(\mathcal I_X/\mathcal I_X^2)^{\vee}$
denote the ideal sheaf and the normal sheaf of $X$ in $V$, respectively.
The symbol $[X]$ represents the point of $\Hilb V$ corresponding to $X$.
Then the tangent space of $\Hilb V$ at $[X]$ is 
known to be isomorphic to the group
$\Hom(\mathcal I_X,\mathcal O_X)$ of sheaf homomorphisms from
$\mathcal I_X$ to $\mathcal O_X$,
which is isomorphic to $H^0(X,N_{X/V})$.
Every obstruction to deforming $X$ in $V$ is contained in
the group $\Ext^1_V(\mathcal I_X,\mathcal O_X)$, 
and if $X$ is a locally complete intersection in $V$, 
then it is contained in a smaller subgroup
$H^1(X,N_{X/V}) \subset \Ext^1_V(\mathcal I_X,\mathcal O_X)$.
If $H^1(X,N_{X/V})=0$, 
then $\Hilb V$ is nonsingular at $[X]$ of dimension $h^0(X,N_{X/V})$.
A {\em first order (infinitesimal) deformation} of $X$ in $V$ 
is a closed subscheme $X' \subset X \times \Spec D$, 
flat over the ring $D=k[t]/(t^2)$ of dual numbers, with a
central fiber $X'_0=X$.
By the universal property of the Hilbert scheme,
there exists a one-to-one correspondence between
the set of $D$-valued points $\gamma: \Spec D \rightarrow \Hilb V$
sending $0$ to $[X]$, and the set of 
the first order deformations of $X$ in $V$.
By the infinitesimal lifting property of smoothness
(cf.~\cite[Proposition~4.4, Chap.~1]{Hartshorne10}),
if there exists a first order deformation of $X$ in $V$
not liftable to a deformation over $\Spec k[t]/(t^n)$ for some $n \ge 3$,
then $\Hilb V$ is singular at $[X]$.
We say $X$ is {\em unobstructed} (resp. {\em obstructed}) in $V$
if $\Hilb V$ is nonsingular (resp. singular) at $[X]$,
and for an irreducible closed subset $W$ of $\Hilb V$,
we say $\Hilb V$ is {\em generically smooth}
(resp. {\em generically non-reduced}) along $W$
if $\Hilb V$ is nonsingular (resp. singular) at the generic point 
$X_{\eta}$ of $W$.

\subsection{Primary obstructions}
\label{subsec:primary}

Let $V$ be a (projective) scheme over $k$, 
$X$ a closed subscheme of $V$,
$\alpha$ a global section of $N_{X/V}$.
We define a cup product 
$\ob(\alpha) \in \Ext^1_V(\mathcal I_X,\mathcal O_X)$ by
$$
\ob(\alpha):= \alpha \cup \mathbf e \cup \alpha,
$$
where $\mathbf e \in \Ext^1_V(\mathcal O_X,\mathcal I_X)$ 
is the extension class of the 
standard short exact sequence
\begin{equation}
  \label{ses:standard}
  0 \longrightarrow \mathcal I_X \longrightarrow \mathcal O_V
  \longrightarrow \mathcal O_X \longrightarrow 0
\end{equation}
on $V$.
Though the following fact is well-known to the experts,
we give a proof for the reader's convenience.
\begin{thm}[{cf.~\cite{Curtin,Kollar}}]
  \label{thm:original cup product}
  Let $\tilde X$ be a first order deformation of $X$ in $V$
  corresponding to $\alpha$. 
  If $X$ is a locally complete intersection in $V$,
  then $\tilde X$ lifts to a deformation ${\tilde {\tilde X}}$
  over $k[t]/(t^3)$,
  if and only if $\ob(\alpha)$ is zero.
\end{thm}

\begin{proof}
First we fix some notations for the proof.
Let $\mathfrak U:=\left\{U_i \bigm| i \in I\right\}$ 
be an open affine covering of $V$,
$R_i$ the coordinate ring of $U_i$,
$I_i$ the defining ideal of $X\cap U_i$ in $U_i$.
We take a covering $\mathfrak U$ such that for all $i,j$,
(i) the intersections $U_{ij}:=U_i \cap U_j$ are affine, and
(ii) $I_i$ are generated by $m$ elements $f_{i1},\dots,f_{im}$ in $R_i$,
where $m$ denotes the codimension of $X$ in $V$.
(Such covering exists by assumption.)
Let $R_{ij}$ be the coordinate ring of $U_{ij}$.
Then since $I_i$ and $I_j$ agree on the overlap $U_{ij}$,
there exists a $m\times m$ matrix $A_{ij}$
with entries in $R_{ij}$ 
(i.e., $A_{ij} \in M(m,R_{ij})$) such that
\begin{equation}
  \label{eqn:transition1}
  \mathbf f_j
  =A_{ij} \mathbf f_i, \quad 
  \mbox{where $\mathbf f_i:=
    \begin{pmatrix}
      f_{i1} \\ \vdots \\ f_{im}
    \end{pmatrix}
    $}.
\end{equation}
Here and later, for a ring $R$,
we denote by $M(m,R)$ the set of $m\times m$ matrices with entries in $R$.
For an object $o$ in $V$,
we denote by $\overline o$ the restriction of $o$ to $X$.
For example, if $u$ is a section of a sheaf $\mathcal F$ on $V$,
$\overline u$ denotes the image of $u$ in 
$\mathcal F\big{\vert}_X=\mathcal F\otimes_{\mathcal O_V} \mathcal O_X$.
We note that the restriction $\overline{A_{ij}}$ to $X$ of $A_{ij}$
represents the transition matrix of $N_{X/V}$ over $U_{ij}$.

Secondly we recall the correspondence between $\tilde X$ and $\alpha$.
Since $X$ is a locally complete intersection in $V$,
so is $\tilde X$ in $V \times \Spec k[t]/(t^2)$ (cf.~\cite[\S9]{Hartshorne10}).
Then for each $i$, the defining ideal $J_i$ of $\tilde X$
over $U_i\times \Spec k[t]/(t^2)$ is generated by
$$
f_{i1}+tu_{i1},\quad \dots, \quad f_{im}+tu_{im}
$$
in $R_i[t]/(t^2)$ for some $u_{ik} \in R_i$ ($k=1,\dots,m$).
Since $J_i$ and $J_j$ agree on $U_{ij}\times \Spec k[t]/(t^2)$,
there exists a matrix $B_{ij} \in M(m,R_{ij})$
such that
$$
\mathbf f_j + t\mathbf u_j
=(A_{ij}+tB_{ij})(\mathbf f_i + t\mathbf u_i),
\qquad
\mbox{where $\mathbf u_i:=
  \begin{pmatrix}
    u_{i1} \\ \vdots \\ u_{im}
  \end{pmatrix}
  $}.
$$
Comparing the coefficient of $t$, we have 
\begin{equation}
  \label{eqn:transition2}
  \mathbf u_j=A_{ij}\mathbf u_i+B_{ij}\mathbf f_i,
\end{equation}
which implies that
$\overline{\mathbf u_j}=\overline{A_{ij}\mathbf u_i}$
in $\mathcal O_{X_{ij}}^{\oplus m}$.
Let $\alpha_i$ be the section of
$N_{X/V}\simeq \mathcal Hom(\mathcal I_X,\mathcal O_X)$ over $U_i$
sending each $f_{ik} \in I_i$ 
to $\overline{u_{ik}}\in R_i/I_i$ ($k=1,\dots,m$),
respectively.
Then by \eqref{eqn:transition1} and \eqref{eqn:transition2},
the local sections $\alpha_i$ ($i \in I$) over $U_i$
agree on $U_{ij}$ for every $i,j$
and define a global section of $N_{X/V}$,
which is nothing but $\alpha$.
In the rest of the proof, for convenience,
we write as $\alpha_i(\mathbf f_i)=\overline{\mathbf u_i}$
instead of writing $\alpha_i(f_{ik})=\overline{u_{ik}}$ ($k=1,\dots,m$).

Now we consider liftings of $\tilde X$
to a second order deformation $\tilde{\tilde X}$ of $X$ in $V$ 
(over $k[t]/(t^3)$). If there exists such a ${\tilde{\tilde X}}$,
then its defining ideal $K_i$ 
over $U_i\times \Spec k[t]/(t^3)$ is generated by
$$
f_{i1}+tu_{i1}+t^2v_{i1},\quad \dots, \quad f_{im}+tu_{im}+t^2v_{im}
$$
in $R_i[t]/(t^3)$ for some $v_{ik} \in R_i$ ($k=1,\dots,m$).
Then there exists a matrix $C_{ij} \in M(m,R_{ij})$
such that
$$
\mathbf f_j + t\mathbf u_j +t^2\mathbf v_j
=(A_{ij}+tB_{ij}+t^2C_{ij})(\mathbf f_i + t\mathbf u_i +t^2\mathbf v_i),
\qquad
\mbox{where $\mathbf v_i:=
  \begin{pmatrix}
    v_{i1} \\ \vdots \\ v_{im}
  \end{pmatrix}
  $},
$$
which is equivalent to that
\begin{equation}
  \label{eqn:cohomologous to 0}
  \overline{\mathbf v_j-A_{ij}\mathbf v_i}
  =\overline{B_{ij}\mathbf u_i}
\end{equation}
in $\mathcal O_{X_{ij}}^{\oplus m}$ 
by comparison of the coefficient of $t^2$.
We see that 
$\tilde {\tilde X}$ is defined as a subscheme of $V \times k[t]/(t^3)$, 
flat over $k[t]/(t^3)$ if and only if
we can solve the equation \eqref{eqn:cohomologous to 0} for $\mathbf v_i$.
On the other hand,
let us define a $1$-cochain
$\beta:=\left\{\beta_{ij}\right\} \in C^1(\mathfrak U,N_{X/V})$,
where $\beta_{ij}$ is the section of 
$N_{X/V}\simeq \mathcal Hom(\mathcal I_X,\mathcal O_X)$
over $U_{ij}$ with
\begin{equation}
  \label{eqn:gamma}
  \beta_{ij}(\mathbf f_i)
  =\overline{B_{ij}\mathbf u_i}.
\end{equation}
Then \eqref{eqn:cohomologous to 0}
implies that $\beta$ is cohomologous to zero,
since $\overline{A_{ij}}$ is the transition matrix of $N_{X/V}$
over $U_{ij}$.
In fact, if we have \eqref{eqn:cohomologous to 0},
then $\beta$ is equal to the coboundary 
of the $0$-cochain $\alpha'=\left\{\alpha'_i\right\}
\in C^0(\mathfrak U,N_{X/V})$
defined by $\alpha'_i(\mathbf f_i)=\overline{\mathbf v_i}$.
Thus for the proof, it suffices to prove the next claim.
\begin{claim}
  The cohomology class in $H^1(X,N_{X/V})$
  represented by $\beta$ equals
  $\ob(\alpha)$.
\end{claim}

\paragraph{\bf Proof of Claim.}\quad
The functor $\Hom(\mathcal I_X,*)$
induces a coboundary map  $\delta: \Hom(\mathcal I_X,\mathcal O_X) \rightarrow
\Ext^1_V(\mathcal I_X,\mathcal I_X)$.
We also deduce from \eqref{ses:standard}
an exact sequence of {\v C}ech complexes
$$
0 \longrightarrow
C^\bullet(\mathfrak U,\mathcal Hom(\mathcal I_X,\mathcal I_X)) \longrightarrow
C^\bullet(\mathfrak U,\mathcal Hom(\mathcal I_X,\mathcal O_V)) \longrightarrow
C^\bullet(\mathfrak U,\mathcal Hom(\mathcal I_X,\mathcal O_X)) \longrightarrow
0.
$$
We compute the image $\delta(\alpha) (=\alpha \cup \mathbf e)$
of $\alpha$ by a diagram chase.
Let $\alpha_i:=\alpha\big{\vert}_{U_i}$ for $i\in I$.
Then as we see before, we have
$\alpha_i(\mathbf f_i)=\overline{\mathbf u_i}$.
If we define a section $\hat \alpha_i$ of 
$\mathcal Hom(\mathcal I_X,\mathcal O_V)$ over $U_i$
by $\hat \alpha_i(\mathbf f_i)=\mathbf u_i$,
then $\hat \alpha_i$ is a local lift of $\alpha$ over $U_i$.
Since $\alpha$ is globally defined, 
$\delta(\alpha)_{ij}=\tilde \alpha_j - \tilde \alpha_i$ 
becomes a section of $\mathcal Hom(\mathcal I_X,\mathcal I_X)$
over $U_{ij}$ for every $i,j$.
Then by \eqref{eqn:transition2},
we have $\delta(\alpha)_{ij}(\mathbf f_i)=B_{ij}\mathbf f_i$.
Thus we have computed $\delta(\alpha)$ as an element of
$H^1(V,\mathcal Hom(\mathcal I_X,\mathcal I_X))
\subset \Ext^1_V(\mathcal I_X,\mathcal I_X)$.
Since 
$\ob(\alpha)=\alpha \cup \mathbf e \cup \alpha=\delta(\alpha)\cup \alpha$,
$\ob(\alpha)$ is represented by the $1$-cocycle
$\left\{\alpha_i \circ\delta(\alpha)_{ij}\right\}$ of $N_{X/V}$.
Therefore $\ob(\alpha)$
is contained in $H^1(X,N_{X/V}) \subset \Ext^1_V(\mathcal I_X,\mathcal O_X)$.
Since we have
$$
\alpha_i \circ\delta(\alpha)_{ij}(\mathbf f_i)
=\alpha_i (B_{ij} \mathbf f_i)
=B_{ij} \alpha_i (\mathbf f_i)
=\overline{B_{ij} \mathbf u_i},
$$
we conclude that
$\alpha_i \circ\delta(\alpha)_{ij}=\beta_{ij}$ 
by \eqref{eqn:gamma}.
Thus we have proved the claim and have finished the proof of 
Theorem~\ref{thm:original cup product}.
\end{proof}

\begin{dfn}
  Here $\ob(\alpha)$ is 
  called the {\em (primary) obstruction} for $\alpha$ 
  (or $\tilde X$). 
\end{dfn}

\subsection{Hypersurface case and $d$-map}
\label{subsec:hypersurface}

Let $X$ be an effective Cartier divisor on $V$,
i.e., a closed subscheme of $V$
whose ideal sheaf is locally generated by a single equation.
We denote by $\Hilb^{cd} V$ the Hilbert scheme of
effective Cartier divisors on $V$.
There exists a natural morphism 
$\varphi: \Hilb^{cd} V \rightarrow \Pic V$
to the Picard scheme $\Pic V$ of $V$,
sending a divisor $D$ on $V$ to the invertible sheaf $\mathcal O_V(D)$
associated to $D$.
We define a morphism
$\psi_X: \Hilb^{cd} V \rightarrow \Pic X$ by the composition
of $\varphi$ with the morphism
$\Pic V \overset{{\vert}_X}\longrightarrow \Pic X$
defined by the restriction to $X$.
By definition, the tangent map $d_X$ of $\psi_X$ at $[X]$
is the composite
\begin{equation}
  \label{map:divisorial d-map}
  d_X: H^0(X,N_{X/V}) \overset{\delta}\longrightarrow H^1(V,\mathcal O_V)
  \overset{|_X}\longrightarrow H^1(X,\mathcal O_X),
\end{equation}
where $\delta$ is the coboundary map of the exact sequence
$0 \rightarrow \mathcal O_V \rightarrow
\mathcal O_V(X) \rightarrow N_{X/V}\rightarrow 0$.
We call $d_X$ the {\em $d$-map} for $X$.
Let $\tilde X$ be a first order deformation of $X$ in $V$,
corresponding to a global section $\beta$ of $N_{X/V}$.
Then by \cite[Lemma 2.9]{Nasu4},
the primary obstruction $\ob(\beta)$ of $\tilde X$
equals the cup product $d_X(\beta)\cup \beta$
by the map
$
H^1(X,\mathcal O_X)\times H^0(X,N_{X/V}) \overset{\cup}{\rightarrow}
H^1(X,N_{X/V}).
$

\subsection{Exterior components}
\label{subsec:exterior}

We recall the definition of the exterior components
(cf.~\cite{Mukai-Nasu,Nasu4}), which is useful for 
computing the obstructions to deforming 
subschemes of codimension greater than $1$.
Let $X$ and $Y$ be two closed subschemes of $V$ such that $X \subset Y$,
$\pi_{X/Y}: N_{X/V} \rightarrow N_{Y/V}\big{\vert}_X$
the natural projection. Then we have the induced maps
$H^i(\pi_{X/Y}): H^i(X,N_{X/V}) \rightarrow H^i(X,N_{Y/V}\big{\vert}_X)$
on the cohomology groups for $i=0,1$. The two images
$$
\pi_{X/Y}(\alpha):=H^0(\pi_{X/Y})(\alpha) 
\qquad
\mbox{and}
\qquad
\ob_Y(\alpha):=H^1(\pi_{X/Y})(\ob(\alpha)) 
$$
are called the {\em exterior component} 
of $\alpha$ and $\ob(\alpha)$, respectively.
These objects respectively correspond to the deformation of $X$ in $V$ into the normal direction to $Y$ and its obstruction.
Now we assume that
$Y$ is an effective Cartier divisor on $V$, and
$X$ is a locally complete intersection in $Y$.
The $d$-map $d_X$ in \eqref{map:divisorial d-map} is generalized
and defined for a pair $(X,Y)$. (In fact, we have $d_Y=d_{Y,Y}$.)

\begin{dfn}
  \label{dfn:d-map}
  Let $\delta$ be the coboundary map of
  $[0 \rightarrow \mathcal I_X \rightarrow
    \mathcal O_V \rightarrow \mathcal O_X \rightarrow 0]\otimes 
  \mathcal O_V(Y)$.
  The composition
  $$
  d_{X,Y}:
  H^0(X,N_{Y/V}\big{\vert}_X) \overset{\delta} \longrightarrow
  H^1(V,\mathcal I_X \otimes \mathcal O_V(Y)) \overset{|_X}\longrightarrow
  H^1(X,{N_{X/V}}^{\vee} \otimes N_{Y/V}\big{\vert}_X)
  $$
  of $\delta$ with the restriction map $|_X$ to $X$
   is called the {\em $d$-map} for $(X,Y)$. 
\end{dfn}

Then the two $d$-maps $d_{X,Y}$ and $d_Y$ are related 
by the following commutative diagram:
\begin{equation}\label{diag:d-map}
  \begin{array}{ccccc}
    H^0(Y,N_{Y/V}) & \mapright{d_Y} & 
     H^1(Y,\mathcal O_Y) \\
    & & \mapdown{|_X}\\
    \mapdown{|_X} && H^1(X,\mathcal O_X) \\
    && \mapdown{H^1(\iota)} \\
    H^0(X,N_{Y/V}\big{\vert}_X) & \mapright{d_{X,Y}} & 
     H^1(X,{N_{X/V}}^{\vee}\otimes N_{Y/V}\big{\vert}_X),
  \end{array}
\end{equation}
where $\iota: \mathcal O_X \rightarrow {N_{X/V}}^{\vee}\otimes N_{Y/V}\big{\vert}_X$
is the sheaf homomorphism induced by ${\pi_{X/Y}}$.  

\begin{lem}[{cf.~\cite[Lemma~2.3 and 2.4]{Mukai-Nasu}}]
  \label{lem:exterior}
  Let $\tilde X$ and $\tilde Y$ 
  be first order deformations of $X$ and $Y$ in $V$, 
  with the corresponding global sections $\alpha$ and $\beta$ 
  of $N_{X/V}$ and $N_{Y/V}$, respectively.
  If we have 
  $\pi_{X/Y}(\alpha) = \beta\big{\vert}_X$
  in $H^0(X,N_{Y/V}\big{\vert}_X)$,
  then we have
  $$
  \ob_Y(\alpha) = d_{X,Y}(\pi_{X/Y}(\alpha))\cup_1 \alpha\\
  = d_Y(\beta)\big{\vert}_X \cup_2 \pi_{X/Y}(\alpha)
  $$
  in $H^1(X,N_{Y/V}\big{\vert}_X)$,
  where $\cup_1$ and $\cup_2$ are the cup product maps
  \begin{align*}
    H^1(X,{N_{X/V}}^{\vee}\otimes N_{Y/V}\big{\vert}_X) \times H^0(X,N_{X/V})
    \overset{\cup_1}{\longrightarrow} H^1(X,N_{Y/V}\big{\vert}_X), \\
    H^1(X,\mathcal O_X)\times H^0(X,N_{Y/V}\big{\vert}_X) 
    \overset{\cup_2}{\longrightarrow}
    H^1(X,N_{Y/V}\big{\vert}_X)    
  \end{align*}
  respectively. 
\end{lem}

\subsection{Infinitesimal deformations with poles and polar $d$-maps}
 \label{subsec:with pole}

In this section, we recall the theory of
{\em infinitesimal deformations with poles},
which was introduced in \cite{Mukai-Nasu}.
Here we develop the study \cite[\S2.4]{Mukai-Nasu} 
on the polar $d$-maps further.
The infinitesimal deformations with poles
are defined as rational sections of some sheaves
admitting a pole along some divisor,
and they are usually regarded as the deformations of 
the open objects complementary to the poles (cf.~\cite{Nasu4}).

Let $V$ be a projective scheme, 
$Y$ and $E$ effective Cartier divisors on 
$V$ and $Y$, respectively.
Put $Y^{\circ}:=Y \setminus E$ and $V^{\circ}:=V \setminus E$ and
let $\iota: Y^{\circ} \hookrightarrow Y$ be the open immersion.
Since $\iota_*\mathcal O_{Y^{\circ}}$ contains 
$\mathcal O_Y(mE)$ as a subsheaf for any $m \ge 0$, 
there exist natural inclusions
$
\mathcal O_Y \subset \mathcal O_Y(E) \subset \cdots
\subset \mathcal O_Y(mE) \subset \cdots \subset \iota_*\mathcal O_{Y^{\circ}}
$
of sheaves on $Y$.
Similarly, since $N_{Y/V}(mE) \subset \iota_* N_{Y^{\circ}/V^{\circ}}$
for any $m \ge 0$,
we regard $H^0(Y,N_{Y/V}(mE))$
as a subgroup of $H^0(Y^{\circ},N_{Y^{\circ}/V^{\circ}})$
by the natural injective map
$$
H^0(Y,N_{Y/V}(mE)) \hookrightarrow 
H^0(Y^{\circ}, N_{Y^{\circ}/V^{\circ}}).
$$
A rational section $\beta$ of $N_{Y/V}$
admitting a pole along $E$, i.e.
$$
\beta \in H^0(Y,N_{Y/V}(mE))
$$ 
for some integer $m\ge 1$ is called an 
{\em infinitesimal deformation of $Y$ with a pole} along $E$.
Every infinitesimal deformation of $Y$ in $V$ 
with a pole induces a first order
deformation of $Y^{\circ}$ in $V^{\circ}$ by the above injection.

Now we assume that the natural map
\begin{equation}
  \label{map:inclusion of H^1}
  H^1(Y,\mathcal O_Y(mE)) \longrightarrow H^1(Y,\mathcal O_Y((m+1)E))
\end{equation}
is injective for any integer $m \ge 1$.
Then since $V$ is projective,
by the same argument as in \cite[Lemma~2.5]{Mukai-Nasu},
the natural map
$$
H^1(Y, \mathcal O_Y(mE))
\longrightarrow  H^1(Y^{\circ}, \mathcal O_{Y^{\circ}})
$$
is injective. 
By this map, we regard $H^1(Y,\mathcal O_Y(mE))$ as a subgroup of
$H^1(Y^{\circ},\mathcal O_{Y^{\circ}})$.
Given an invertible sheaf  $L$ on $Y$, 
we identify an element of $H^1(Y,\mathcal O_Y(mE))$ 
as a first order deformation of
the invertible sheaf $L^{\circ}:=\iota_* L$ on $Y^{\circ}$,
and call it an {\em infinitesimal deformation of $L$ with a pole} along $E$. 

Let $m \ge 1$ be an integer and $\beta \in H^0(Y,N_{Y/V}(mE))$ 
an infinitesimal deformation of $Y$ with a pole along $E$.
Let $d_{Y^{\circ}}: H^0(Y^{\circ},N_{Y^{\circ}/V^{\circ}})
\rightarrow H^1(Y^{\circ},\mathcal O_{Y^{\circ}})$ be the $d$-map
\eqref{map:divisorial d-map}
for $Y^{\circ} \subset V^{\circ}$.
The following is a generalization of
\cite[Proposition 2.6]{Mukai-Nasu}, which enables us to
compute the singularity of
$d_{Y^{\circ}}(\beta) \in H^1(Y^{\circ},\mathcal O_{Y^{\circ}})$ 
along the boundary $E$. 

\begin{prop}\label{prop:key}
  Let $m \ge 1$ be an integer. Then
  \begin{enumerate}
    \item $d_{Y^{\circ}}(H^0(Y,N_{Y/V}(mE)))
    \subset H^1(Y,\mathcal O_Y((m+1)E)).$
    \item Let $d_Y$ be the restriction of $d_{Y^{\circ}}$
    to $H^0(Y,N_{Y/V}(mE))$,
    and let 
    $\partial_E$ be the coboundary map of \eqref{ses:normal bundle of E}.
    Then the diagram
    \begin{equation*}
      \begin{CD}
	H^0(Y,N_{Y/V}(mE)) @>d_{Y}>> H^1(Y,\mathcal O_Y((m+1)E)) \\ 
	@V{|_E}VV @V{|_E}VV \\
	H^0(E,N_{Y/V}(mE)\big{\vert}_E) @>{m\partial_E}>>
	H^1(E,\mathcal O_E((m+1)E))
      \end{CD}
    \end{equation*}
    is commutative.
  \end{enumerate}
\end{prop}

In other words, if $Y$ is a hypersurface in $V$,
then every infinitesimal deformation of $Y \subset V$ with a pole
induces that of the invertible sheaf $N_{Y/V}$.
The principal part of $d_{Y^\circ}(\beta)$ along $E$
coincides with the coboundary $\partial_E(\beta\big{\vert}_E)$
of the principal part $\beta\big{\vert}_E$,
up to constant.

\begin{proof}
The proof is similar to the one in \cite{Mukai-Nasu}, where
$Y$ is a surface by assumption.
Let $\mathfrak U:=\{ U_i \}_{i \in I}$ be an open affine covering of $V$ 
and let $x_i=y_i=0$ be the local equation of $E$ over $U_i$ such that
$y_i$ defines $Y$ in $U_i$.
Through the proof, for a local section $t$ of a sheaf $\mathcal F$ on $V$, 
$\bar t$ denotes the restriction 
$t\big{\vert}_Y \in \mathcal F\big{\vert}_Y$ for conventions.
Let $D_{x_i}$ and $D_{\bar x_i}$ denote the open affine subsets of 
$U_i$ and $U_i \cap Y$ defined by $x_i \ne 0$ and $\bar x_i \ne 0$,
respectively.
Then 
$\left\{ D_{\bar x_i} \right\}_{i \in I}$ is an open affine covering of 
$Y^{\circ}$, since $D_{\bar x_i}=D_{x_i} \cap Y=U_i \cap Y^{\circ}$.

Let $\beta$ be a global section of $N_{Y/V}(mE) \simeq \mathcal O_Y(Y)(mE)$.
Then the product $\bar x_i^m \beta$ is contained in $H^0(U_i,\mathcal O_Y(Y))$
and lifts to a section
$s_i' \in \Gamma(U_i, \mathcal O_V(Y))$ since $U_i$ is affine.
In particular,
$\beta$ lifts to the section $s_i:=s_i'/x_i^m$ of 
$\mathcal O_{V^{\circ}}(Y^{\circ})$ over $D_{x_i}$.
Let $\delta: H^0(Y^{\circ},N_{Y^{\circ}/V^{\circ}})
\rightarrow H^1(Y^{\circ},\mathcal O_{V^{\circ}})$ 
be the coboundary map
in the definition \eqref{map:divisorial d-map} of $d_{Y^{\circ}}$.
Then we have
$$
\delta(\beta)_{ij}= s_j-s_i
=\dfrac{s_j'}{x_j^m} -\dfrac{s_i'}{x_i^m}
\qquad
\mbox{in}
\qquad
\Gamma(D_{x_i} \cap D_{x_j}, \mathcal O_{V^{\circ}}(Y^{\circ}))
$$
for every $i,j$.
Since $\beta$ is a global section of $N_{Y^{\circ}/V^{\circ}}$,
$\delta(\beta)_{ij}$ is contained in 
$\Gamma(D_{x_i} \cap D_{x_j}, \mathcal O_{V^{\circ}})$.
Now we put
$$
f_{ij}:=x_i^mx_j^m\delta(\beta)_{ij}=x_i^ms_j'-x_j^ms_i'.
$$
Since $x_i^ms_i=s_i' \in \Gamma(U_i,\mathcal O_V(Y))$ 
for every $i$, $f_{ij}$ is a section of $\mathcal O_{U_{ij}}$.
Now we recall the relations between the
local equations $x_i,y_i$ of $E$ over $U_i$.
Since the two ideals $(x_i,y_i)$ and $(x_j,y_j)$ agree
on the overlap $U_{ij}$,
there exist elements
$b_{ij}$ and $c_{ij}$ of $\mathcal O_{U_{ij}}$ satisfying
$x_i = b_{ij} y_j + c_{ij} x_j$.
Then we have
\begin{align}
  f_{ij}&=(x_i^m-c_{ij}^mx_j^m)s_j'+x_j^m(c_{ij}^ms_j'-s_i')\notag \\
   &=\sum_{k=0}^{m-1}x_i^{m-1-k}(c_{ij}x_j)^k(x_i-c_{ij}x_j)s_j'
  +x_j^m(c_{ij}^ms_j'-s_i')\notag \\
  &=\sum_{k=0}^{m-1}x_i^{m-1-k}(c_{ij}x_j)^kb_{ij}y_js_j'
  +x_j^m(c_{ij}^ms_j'-s_i')\label{eqn:f_ij}.
\end{align}
Since $y_j \in \Gamma(U_j,\mathcal O_V(-Y))$ and
$s_j' \in \Gamma(U_j,\mathcal O_V(Y))$,
$b_{ij}y_js_j'$ is a section of $\mathcal O_{U_{ij}}$,
while $c_{ij}^ms_j'-s_i' \in \Gamma(U_{ij},\mathcal O_V(Y))$ 
is also a section of $\mathcal O_{U_{ij}}$,
because we have
$$
\ov{c_{ij}^ms_j'-s_i'}
=\bar c_{ij}^m\bar x_j^m\beta-\bar x_i^m\beta
=(\bar c_{ij}^m\bar x_j^m-\bar x_i^m)\beta=0
$$
in $\Gamma(Y\cap U_{ij},N_{Y/V})$.
Therefore, $f_{ij}$ is contained in 
$\Gamma(U_{ij},{\mathcal I_E}^{m-1})$ by \eqref{eqn:f_ij},
and hence $\bar f_{ij}$ is contained in
$\Gamma(Y\cap U_{ij},\mathcal O_Y((-m+1)E))$.
This implies that
\begin{equation}
 \label{eqn:d_Y}
 d_{Y^{\circ}}(\beta)_{ij} =(\delta(\beta)_{ij})\big{\vert}_{Y^\circ}= 
  \dfrac{\bar f_{ij}}{\bar x_i^m \bar x_j^m}
  \qquad
  \mbox{in}
  \qquad
  \Gamma(D_{\bar x_i} \cap D_{\bar x_j}, \mathcal O_{Y^{\circ}}), 
\end{equation}
is contained in 
$\Gamma(Y\cap U_{ij}, \mathcal O_Y((m+1)E))$.
Thus we have proved (1).

Now we compute the image of $d_Y(\beta)=d_{Y^{\circ}}(\beta)$ by
the restriction map 
$H^1(Y,\mathcal O_Y((m+1)E)) \rightarrow H^1(E,\mathcal O_E((m+1)E))$,
regarding $\mathcal O_E((m+1)E)$ 
as the quotient sheaf $\mathcal O_Y((m+1)E)/\mathcal O_Y(mE)$.
Since $\ov x_i=\ov{c_{ij}x_j}$ and $\beta=\ov s_j'/\ov x_j^m$,
it follows from \eqref{eqn:f_ij} and \eqref{eqn:d_Y} that
$$
d_Y(\beta)_{ij}
= \sum_{k=0}^{m-1}\left(
  \dfrac{\ov{c_{ij}x_j}}{\ov x_i}
\right)^k
\dfrac{\overline{b_{ij}y_j}}{\ov x_i}
\dfrac{\ov s_j'}{\ov x_j^m}
+\dfrac{\ov{c_{ij}^ms_j'-s_i'}}{\ov x_i^m}
= m\dfrac{\overline{b_{ij}y_j}}{\ov x_i}\beta
+\dfrac{\ov{c_{ij}^ms_j'-s_i'}}{\ov x_i^m}
$$
in $\Gamma(Y\cap U_{ij}, \mathcal O_Y((m+1)E))$.
Since $c_{ij}^ms_j' - s_i'$ a section of $\mathcal O_{U_{ij}}$,
$\ov{c_{ij}^ms_j'-s_i'}/\bar x_i^m$ is contained in 
$\Gamma(Y\cap U_{ij},\mathcal O_Y(mE))$.
On the other hand, the restriction of the $1$-cochain 
$\left\{ \ov{b_{ij}y_j}/\bar x_i \right\}_{i,j \in I}$ to $E$ is a cocycle and 
represents the extension class $\mathbf e' \in H^1(E,\mathcal O_E(-Y+E))$ 
of the exact sequence \eqref{ses:normal bundle of E} (cf.~\cite{Mukai-Nasu}).
Therefore $d_Y(\beta)\big{\vert}_E$ is equal to
$m\mathbf e' \cup (\beta\big{\vert}_E)=m\partial_E(\beta\big{\vert}_E)$,
which implies (2).
\end{proof}

We finish this section by giving a refinement of Proposition~\ref{prop:key},
which will be used in the proof of Theorem~\ref{thm:refinement}.
Let $E_i$ ($1\le i\le k$) be irreducible Cartier divisors on $Y$.
Suppose that $E_i$ are mutually disjoint,
i.e., $E_i \cap E_j=\emptyset$ for all $i,j$.
We suppose furthermore that
for any two effective divisors $D,D'$ on $S$
with supports on $\bigcup_{i=1}^k E_i$,
if $D \le D'$, then the natural map
$H^1(Y,\mathcal O_Y(D)) \rightarrow H^1(Y,\mathcal O_Y(D'))$
is injective. 
Then as we have seen before, for any such divisor $D$, 
$H^1(Y,\mathcal O_Y(D))$ is regarded as a subgroup of
$H^1(Y^{\circ},\mathcal O_{Y^{\circ}})$, 
where $Y^{\circ}:=Y\setminus \bigcup_{i=1}^k E_i$.
Let $E=\sum_{i=1}^k m_i E_i$ be an effective divisor on $Y$
with coefficients $m_i \ge 1$, and let $\beta \in H^0(Y,N_{Y/V}(E))$.
We put $V^{\circ}:=V\setminus \bigcup_{i=1}^k E_i$ and 
denote by $d_{Y^{\circ}}$ the $d$-map 
\eqref{map:divisorial d-map} for $Y^{\circ} \subset V^{\circ}$.
If $H^1(Y,N_{Y/V})=0$, then by the following lemma,
$\beta\in H^0(Y,N_{Y/V}(E))$
is written as a $k$-linear combination $\sum_{i=1}^k c_i \beta_i$
of $\beta_i \in H^0(Y,N_{Y/V}(m_iE_i))$.

\begin{lem}
  \label{lem:k-linear combination}
  Let $L$ be an invertible sheaf on $Y$ with $H^1(Y,L)=0$,
  and let $E,E'$ be two effective divisors on $Y$
  whose supports are mutually disjoint. Then the natural map
  $
  H^0(Y,L(E))\oplus H^0(Y,L(E')) \rightarrow H^0(Y,L(E+E'))
  $
  is surjective.
\end{lem}

\begin{proof}
It follows from the exact sequence
$[0 \rightarrow \mathcal O_Y
  \rightarrow \mathcal O_Y(E) \oplus \mathcal O_Y(E')
  \rightarrow \mathcal O_Y(E+E')
  \rightarrow 0]\otimes L$
on $Y$ of Koszul type. 
\end{proof}

Since the $d$-map $d_{Y^{\circ}}$ is $k$-linear, we find
$d_{Y^{\circ}}(\beta)=\sum_{i=1}^k c_i d_{Y^{\circ}}(\beta_i)$.
Because for each $i$, $d_{Y^{\circ}}(\beta_i)$ is contained in 
$H^1(Y,\mathcal O_Y((m_i+1)E_i))$ by 
Proposition~\ref{prop:key2},
$d_{Y^{\circ}}(\beta)$ is contained in 
$H^1(Y,\mathcal O_Y(\sum_{i=1}^k (m_i+1)E_i))$. 
Furthermore, since $E_i$'s are mutually disjoint,
we have $d_{Y^{\circ}}(\beta_i)\big{\vert}_{E_j}=0$ if $i \ne j$
and 
$d_{Y^{\circ}}(\beta_i)\big{\vert}_{E_i}
=m_i\partial_{E_i}(\beta_i\big{\vert}_{E_i})$ 
by the same proposition. Thus we conclude that
\begin{prop}
  \label{prop:key2}
  Let $m_i \ge 1$ be integers. If $H^1(Y,N_{Y/V})=0$, then
  \begin{enumerate}
    \item $d_{Y^{\circ}}(H^0(Y,N_{Y/V}(\sum_{i=1}^k m_i E_i)))
    \subset H^1(Y,\mathcal O_Y(\sum_{i=1}^k (m_i+1)E_i))$.
    \item Let $d_Y$ be the restriction of $d_{Y^{\circ}}$
    to $H^0(Y,N_{Y/V}(\sum_{i=1}^k m_i E_i))$. Then the diagram
    \begin{equation*}
      \begin{CD}
	H^0(Y,N_{Y/V}(\sum_{i=1}^k m_i E_i)) @>d_{Y}>>
	H^1(Y,\mathcal O_Y(\sum_{i=1}^k (m_i+1)E_i)) \\ 
	@V{|_{E_i}}VV @V{|_{E_i}}VV \\
	H^0(E_i,N_{Y/V}(m_iE_i)\big{\vert}_{E_i}) 
	@>{m_i\partial_{E_i}}>>
	H^1(E_i,\mathcal O_{E_i}((m_i+1)E_i))
      \end{CD}
    \end{equation*}
    is commutative for any $i=1,\dots,k$.
  \end{enumerate}
\end{prop}

\subsection{Hilbert-flag schemes}
\label{subsec:flag schemes}

In this section, we recall some basic results on Hilbert-flag schemes.
For the construction (the existence), the local properties, etc.,
of the Hilbert-flag schemes, we refer to
\cite{Kleppe81,Kleppe87,Hartshorne10,Sernesi}.

Let $V$ be a projective scheme, 
and let $X, Y$ be two closed subschemes of $V$ such that $X \subset Y$,
with the Hilbert polynomials $P, Q$, respectively.
Then there exists a projective scheme $\HF_{P,Q} V$,
called the Hilbert-flag scheme of $V$, 
parametrising all pairs $(X',Y')$ of closed subschemes 
$X' \subset Y' \subset V$
with the Hilbert polynomials $P$ and $Q$, respectively.
There exists a natural diagram of the Hilbert(-flag) schemes
$$
\raisebox{20pt}{
\xymatrix{
\HF_{P,Q} V \ar[r]^{pr_2} \ar[d]_{pr_1}  & \Hilb_Q V \\
\Hilb_P V &  
}}
\quad
\left(
\raisebox{20pt}{
\xymatrix{
(X,Y) {\ar@{|->}[r]} \ar@{|->}[d]  & Y \\ 
X &  
}}
\right),
$$
where $pr_i$ ($i=1,2$) are the forgetful morphisms, i.e., the projections.
We denote the tangent space of $\HF V$ at $(X,Y)$ by $A^1(X,Y)$.
Then there exists a Cartesian diagram
\begin{equation}
  \label{diag:cartesian}
  \raisebox{25pt}{
  \xymatrix{
    A^1(X,Y) \ar[d]_{p_1} \ar[r]^{p_2} \ar@{}[dr]|\square & 
    H^0(Y,N_{Y/V}) \ar[d]_{\rho} \\
    H^0(X,N_{X/V}) \ar[r]^{\pi_{X/Y}} & H^0(X,N_{Y/V}\big{\vert}_X) \\
  }}
\end{equation}
where $p_i$ is the tangent map of $pr_i$ ($i=1,2$), 
$\rho$ is the restriction map,
and $\pi_{X/Y}$ is the projection.
In what follows, we assume that $X$ and $Y$ are smooth
and $\Hilb V$ is nonsingular at $[Y]$. Let 
$$
\partial_X: H^0(X,N_{Y/V}\big{\vert}_X)
\rightarrow H^1(X,N_{X/Y})
$$
be the coboundary map of the exact sequence
$0\rightarrow N_{X/Y} \rightarrow N_{X/V}
\overset{\pi_{X/Y}}\rightarrow N_{Y/V}\big{\vert}_X
\rightarrow 0$ on $V$ and 
let $\alpha_{X/Y}$ be the composition $\partial_X \circ \rho$
of $\rho$ with $\partial_X$.
Then since $\Hilb V$ is nonsingular at $[Y]$,
every obstruction to deforming a pair $(X,Y)$
of subschemes $X,Y$ with $X \subset Y \subset V$ is 
contained in the group
\begin{equation}
\label{eqn:obst.sp.flag}
A^2(X,Y):=\coker\alpha_{X/Y}, 
\end{equation}
and we have 
\begin{equation}
  \label{ineq:dimension of flag}
  \dim A^1(X,Y) - \dim A^2(X,Y) \le \dim_{(X,Y)} \HF V \le \dim A^1(X,Y) 
\end{equation}
(cf.~\cite[Theorem 1.3.2]{Kleppe81}, \cite[\S2]{Kleppe87}).
Thus $A^2(X,Y)$ represents the obstruction space of $\HF V$ at $(X,Y)$.
There exists an exact sequence
\begin{align}
  \label{seq:flag to hilb}
  0 &\longrightarrow H^0(Y,\mathcal I_{X/Y}\otimes N_{Y/V}) 
  \longrightarrow A^1(X,Y)
  \longrightarrow H^0(X,N_{X/V})  \notag \\
  &\longrightarrow \coker \rho
  \longrightarrow A^2(X,Y)
  \longrightarrow H^1(X,N_{X/V})  
  \longrightarrow H^1(X,N_{Y/V}\big{\vert}_X) 
\end{align}
of cohomology groups, which connects 
the tangent spaces and the obstruction spaces of Hilbert(-flag) schemes
(see \cite{Kleppe81,Kleppe87} for the proof).
If we have $H^i(Y,N_{Y/V})=0$ for $i=1,2$, then we deduce from 
the exact sequence
$
[0 \rightarrow \mathcal I_{X/Y}
\rightarrow \mathcal O_Y \rightarrow \mathcal O_X
\rightarrow 0]\otimes N_{Y/V}
$
the two isomorphisms 
$\coker \rho \simeq H^1(Y,\mathcal I_{X/Y}\otimes N_{Y/V})$
and 
$H^1(X,N_{Y/V}\big{\vert}_X) \simeq H^2(Y,\mathcal I_{X/Y}\otimes N_{Y/V})$.
If $\dim X=1$ then the last map of \eqref{seq:flag to hilb} is surjective.
Thus we obtain (3) of the next lemma.
\begin{lem}
  \begin{enumerate}
    \item If $\rho$ is surjective (cf.~\eqref{diag:cartesian}),
    then $pr_1: \HF V \rightarrow \Hilb V$ is smooth at $(X,Y)$
    (cf.~\cite[Lemma A10]{Kleppe87}).
    \item If $H^0(Y,\mathcal I_{X/Y}\otimes N_{Y/V})=0$, then
    $pr_1$ is a local embedding at $(X,Y)$.
    \item If $\dim X=1$, $\dim Y=2$ and $H^i(Y,N_{Y/V})=0$ ($i=1,2$),
    then we have
    \begin{align}
      \label{eqn:expected dimension}
      \dim A^1(X,Y)- \dim A^2(X,Y)
      &=\chi(X,N_{X/V})+\chi(Y,\mathcal I_{X/Y}\otimes N_{Y/V}) \notag \\
      &=\chi(X,N_{X/Y})+\chi(Y,N_{Y/V}). 
    \end{align}
  \end{enumerate}
  \label{lem:flag to hilb}
\end{lem}
The number \eqref{eqn:expected dimension}
represents the expected dimension of the Hilbert-flag scheme
$\HF V$ at $(X,Y)$.
If $A^2(X,Y)=0$, then $\HF V$ is nonsingular at $(X,Y)$
by \eqref{ineq:dimension of flag}.
If moreover $H^1(Y,\mathcal I_{X/Y}\otimes N_{Y/V})=0$, then 
so is $\Hilb V$ at $[X]$ by Lemma~\ref{lem:flag to hilb}(1).
The following lemma will be essentially used in 
the proof of Theorem~\ref{thm:main2} (cf.~\S\ref{sect:k3}).

\begin{lem}
  \label{lem:flag of k3fano}
  Let $V$ be a smooth Fano $3$-fold, 
  $S$ a smooth $K3$ surface contained in $V$,
  $C$ a smooth curve on $S$. Then
  \begin{enumerate}
    \item $H^i(S,N_{S/V})=0$ for all $i\ge 1$.
    In particular, $\Hilb V$ is nonsingular at $[S]$.
    \item We have an isomorphism
    \begin{equation}
      \label{isom:cohomology of abnormality}
      H^i(S,\mathcal I_{C/S}\otimes N_{S/V}) \simeq  H^i(S,-D)
    \end{equation}
    for every integer $i$,
    where $D:=C+K_V\big{\vert}_S$ is a divisor on $S$.
    \item Suppose that there exists a first order deformation $\tilde S$ of $S$
    which does not contain any first order deformations $\tilde C$ of $C$.
    Then we have $A^2(C,S)=0$. In particular, 
    the Hilbert-flag scheme $\HF V$ is nonsingular at $(C,S)$
    of dimension $(-K_V\big{\vert}_S)^2/2+g(C)+1$,
    where $g(C)$ is the genus of $C$.
  \end{enumerate}
\end{lem}
\begin{proof}
Since $K_S$ is trivial, by adjunction, 
we have $N_{S/V}\simeq -K_V\big{\vert}_S$ and $N_{C/S}\simeq K_C$.
Then (1) follows from the ampleness of $-K_V$,
and (2) from 
$\mathcal I_{C/S}\otimes N_{S/V} \simeq
N_{S/V}(-C)\simeq -K_V\big{\vert}_S-C = -D$.
On the other hand, we have $H^1(C,N_{C/S})\simeq k$.
Thus the obstruction group $A^2(C,S)$
(cf.~\eqref{eqn:obst.sp.flag}) of $\HF V$ at $(C,S)$
is of dimension at most $1$. 
For proving (3), let $\beta$ be the global section of $N_{S/V}$ corresponding 
to $\tilde S$. Then $\rho(\beta)$ is not contained in the image of 
$\pi_{C/S}$, because the diagram \eqref{diag:cartesian} is Cartesian.
Hence the map $\alpha_{C/S}$ is nonzero and 
we conclude that $A^2(C,S)=0$.
By Lemma~\ref{lem:flag to hilb}(3),
$\dim_{(C,S)} \HF V=\dim A^1(C,S)
=\chi(-K_V\big{\vert}_S)+\chi(K_C)=(-K_V\big{\vert}_S)^2/2+g(C)+1$.
\end{proof}

The first projection $pr_1$ induces a morphism
$pr_1': \HF^{sc} V \rightarrow \Hilb^{sc} V$,
where $\HF^{sc} V:=pr_1^{-1} (\Hilb^{sc} V)$.
If $X$ is a smooth connected curve
and $\HF V$ is nonsingular at $(X,Y)$,
then there exists a unique irreducible component 
$\mathcal W_{X,Y}$ of $\HF^{sc} V$ passing through $(X,Y)$.
\begin{dfn}
  \label{dfn:S-maximal}
  The image $W_{X,Y}$ of $\mathcal W_{X,Y}$ by $pr_1'$ is called
  the {\bf $Y$-maximal family} of curves containing $X$.
\end{dfn}

\subsection{$K3$ surfaces and quartic surfaces}
\label{subsec:K3 and quartic}

We recall some basic results on 
$K3$ surfaces and quartic surfaces.

\begin{lem}
  \label{lem:K3}
  Let $S$ be a smooth projective $K3$ surface, 
  $D\ne 0$ an effective divisor on $S$.
  \begin{enumerate}
    \item If $D$ is nef, then the complete linear system $|D|$ 
    has a base point if and only if 
    there exist curves $E$ and $F$ on $S$ and an integer $k \ge 2$
    such that $D\sim E+kF$, $E^2=-2$, $F^2=0$ and $E.F=1$.
    \item Let $D^2\ge 0$. Then $H^1(S,D)\ne 0$ if and only if
    (i) $D.\Delta \le -2$ for some divisor $\Delta \ge 0$ 
    with $\Delta ^2=-2$, or
    (ii) $D \sim kF$ for some nef and primitive 
    divisor $F\ge 0$ with $F^2=0$ and 
    an integer $k \ge 2$.
    (We have $h^1(S,D)=k-1$ in (ii).)
  \end{enumerate}
\end{lem}

\begin{proof}
(1) follows from \cite[2.7]{Saint-Donat}
and (2) from \cite{Knutsen-Lopez}.
\end{proof}

The following lemma will be used in \S\ref{sect:non-reduced}
to show the existence of quartic surfaces of Picard number two
containing a rational curve or an elliptic curve.

\begin{lem}[{Mori~\cite{Mori84}, see also \cite[p.138]{Hartshorne10}}]
  \label{lem:mori}
  We assume that $\car k=0$.
  \begin{enumerate}
    \item 
    There exists a smooth curve $C$ of degree $d>0$ and genus $g\ge 0$
    on a smooth quartic surface $S \subset \mathbb P^3$ 
    if and only if (i) $g=d^2/8+1$, or (ii) $g<d^2/8$ and $(d,g)\ne (5,3)$.
    \item If there exists a smooth quartic surface $S_0$ 
    containing smooth curve $C_0$ of degree $d$ and genus $g$,
    then there exists a smooth quartic surface $S_1$ containing 
    a smooth curve $C_1$ of the same degree and genus,
    with the property that 
    $\Pic S_1$ is generated by $C_1$ and 
    the class $\mathbf h$ of hyperplane sections.
  \end{enumerate}
\end{lem} 

The following lemma will be used in \S\ref{sect:non-reduced}
to apply Theorem~\ref{thm:main2}.

\begin{lem}
  \label{lem:non-surjective}
  Let $V$ be a smooth projective variety, 
  $S$ a smooth $K3$ surface contained in $V$, 
  and $E$ a smooth curve on $S$. If
  \begin{enumerate}
    \renewcommand{\theenumi}{{\alph{enumi}}}
    \item $V\simeq \mathbb P^n$ and 
    $E$ is rational (i.e. $E \simeq \mathbb P^1$) or elliptic,
    \label{item:projective space}
    \item $E$ is rational and $N_{E/V}$ is globally generated, or
    \label{item:rational and globally generated}
    \item $E$ is elliptic and there exists a first order deformation
    $\tilde S$ of $S$ not containing any first order deformation 
    $\tilde E$ of $E$,
    \label{item:elliptic}
  \end{enumerate}
  then the $\pi$-map $\pi_{E/S}(E)$ in \eqref{map:pi-map}
  is not surjective.
\end{lem}

\begin{proof}
There exists an exact sequence
$$
[0 \longrightarrow N_{E/S} 
\overset{\iota}\longrightarrow N_{E/V}
\overset{\pi_{E/S}}\longrightarrow N_{S/V}\big{\vert}_E
\longrightarrow 0]\otimes \mathcal O_E(E),
$$
on $E$, where $\iota$ is a natural inclusion.
We note that $N_{E/S}\simeq \mathcal O_E(E)\simeq K_E$,
since $K_S$ is trivial.
Thus we have 
$H^1(E,N_{E/S}(E)) \simeq H^1(E,2K_E) \ne 0$ by assumption.
We prove $H^1(E,N_{E/V}(E))=0$ in the case \eqref{item:projective space}.
Suppose that $V\simeq \mathbb P^n$.
Then by the Euler sequence 
$0 \rightarrow \mathcal O_{\mathbb P^n}
\rightarrow \mathcal O_{\mathbb P^n}(1)^{n+1}
\rightarrow \mathcal T_{\mathbb P^n}
\rightarrow 0$ on $\mathbb P^n$,
there exists a surjective homomorphism
$[\mathcal O_E(1)^{n+1}
\twoheadrightarrow 
\mathcal T_{\mathbb P^n}\big{\vert}_E]
\otimes_{\mathcal O_E} \mathcal O_E(E)$.
We have $H^1(E,\mathcal O_E(E)(1))\simeq H^1(E,K_E(1))=0$
and $E$ is a curve, 
we have $H^1(E,\mathcal T_{\mathbb P^n}\big{\vert}_E(E))=0$
and hence $H^1(E,N_{E/\mathbb P^n}(E))=0$.
Suppose that \eqref{item:rational and globally generated} is satisfied.
Then we have an inequality $h^1(E,N_{E/V}(E))\le 2$,
while $h^1(E,N_{E/S}(E))=h^1(\mathbb P^1,\mathcal O_{\mathbb P^1}(-4))=3$.
Therefore $H^1(E,\iota \otimes_{\mathcal O_E} \mathcal O_E(E))$ 
is not injective.
Suppose that \eqref{item:elliptic} is satisfied.
As we see in the proof of Lemma~\ref{lem:flag of k3fano},
by assumption,
$H^0(E,N_{E/V}) \overset{\pi_{E/S}}\rightarrow H^0(E,N_{S/V}\big{\vert}_E)$
is not surjective,
and so is $\pi_{E/S}(E)$,
because $\mathcal O_E(E)$ is trivial.
\end{proof}

\section{Obstructedness criterion}
\label{sect:obstruction}

In this section, we compute obstructions to deforming 
curves on a $3$-fold, and prove Theorem~\ref{thm:main1}
and its refinement Theorem~\ref{thm:refinement}.

Let $C \subset S \subset V$ be a sequence of 
a curve $C$, a surface $S$, a $3$-fold $V$, 
$E$ an effective Cartier divisor on $S$.
We assume that $C$ and $S$ are Cartier divisors on $S$ and $V$,
respectively.
Let $Z:=C\cap E$ be the scheme-theoretic intersection of $C$ and $E$.
In this section, given a coherent sheaf $\mathcal F$ on $S$ (resp. $C$),
integers $i,m \ge 0$, and a cohomology class $*$ in 
$H^i(S,\mathcal F)$ (resp. $H^i(C,\mathcal F)$),
we denote by $\overline *_{(m)}$ the image of
$*$ in $H^i(S,\mathcal F(mE))$ (resp. $H^i(C,\mathcal F(mZ))$).
We define $\mathbf k_C \in \Ext^1_S(\mathcal O_C,\mathcal O_S(-C))$
as the extension class of the exact sequence 
\begin{equation}
  \label{ses:CE}
  0 \longrightarrow \mathcal O_S(-C)
  \longrightarrow \mathcal O_S
  \longrightarrow \mathcal O_C
  \longrightarrow 0,
\end{equation}
and the class $\mathbf k_E$ similarly.

\begin{lem}
  \label{lem:restriction to C and E}
  Let $L$ be an invertible sheaf on $S$ and 
  $\gamma$ a global section of 
  $L\big{\vert}_C=L\otimes_{\mathcal O_S} \mathcal O_C$.
  \begin{enumerate}
    \item Let $m \ge 0$ be an integer. Then
    $\overline \gamma_{(m)}\in H^0(C,L(mE)\big{\vert}_C)$ 
    lifts to a section $\beta \in H^0(S,L(mE))$ on $S$
    if and only if 
    $\overline \gamma_{(m)}\cup \mathbf k_C=0$
    in $H^1(S,L(mE-C))$.
    \item If 
    $\overline \gamma_{(m)}$ lifts to a section 
    $\beta \in H^0(S,L(mE))$ for $m \ge 1$,
    then the principle part $\beta\big{\vert}_E$ of $\beta$
    is contained in $H^0(E,L(mE-C)\big\vert_E)$, and hence
    $\beta$ is contained in 
    $H^0(S,\mathcal I_{Z/S}\otimes L(mE))$.
    Here $\beta\big{\vert}_E$ is nonzero 
    if and only if $\beta \not\in H^0(S,L((m-1)E))$,
    equivalently, 
    $\overline \gamma_{(m-1)}\cup \mathbf k_C \ne 0$
    in $H^1(S,L((m-1)E-C))$.
  \end{enumerate}
  
\end{lem}

\begin{proof}
(1) follows from the short exact sequence $\eqref{ses:CE}\otimes L(mE)$,
whose coboundary map coincides with the cup product map
$\cup \mathbf k_C: H^0(C,L(mE)\big{\vert}_C) \rightarrow H^1(S,L(mE-C))$
with $\mathbf k_C$.
(2) follows from a diagram chase on the commutative diagram
$$
  \begin{array}{ccccc}
    & & 0 & & 0\\
    & & \downarrow & & \downarrow \\
    & & H^0(S,L(mE-C)) & \mapright{|_E} & H^0(E,L(mE-C)\big{\vert}_E) \\
    & & \mapdown{} & & \mapdown{} \\
    H^0(S,L((m-1)E))& \longrightarrow & H^0(S,L(mE)) & \mapright{|_E} & H^0(E,L(mE)\big{\vert}_E)\\
    \mapdown{|_C} & & \mapdown{|_C} & & \mapdown{|_C} \\
    H^0(C,L((m-1)E)\big{\vert}_C) & \longrightarrow & H^0(C,L(mE)\big{\vert}_C) & \mapright{|_E} & k(Z)\\
    \mapdown{\cup \mathbf k_C} & & \mapdown{\cup \mathbf k_C} & &  \\
    H^1(S,L((m-1)E-C))& \longrightarrow & H^1(S,L(mE-C)) &  &  \\
  \end{array}
$$
of cohomology groups, which is exact both vertically and horizontally.
\end{proof}

\paragraph{\bf Proof of Theorem~\ref{thm:main1}.}\quad
We first recall the relation between $\alpha$ 
and $\partial_E(\beta\big{\vert}_E)$ by Figure~\ref{fig:relation}.
\begin{figure}[h]
  $$
  \begin{array}{ccccccccc}
    H^0(N_{C/V}) & \ni & \alpha & & & & & & H^0(N_{E/V}(mE)) \\
    \mapdown{\pi_{C/S}} && \big\downarrow && && &&\mapdown{\pi_{E/S}(mE)} \\
    H^0(N_{S/V}\big{\vert}_C) & \ni & \beta\big{\vert}_C
    & \overset{res}\longmapsfrom & \beta & \overset{res}\longmapsto &
    \beta\big{\vert}_E & \in & H^0(N_{S/V}(mE)\big{\vert}_E) \\
    \bigcap &&&& \rotatebox{90}{$\ni$} && \mapdown{} && \mapdown{\partial_E}\\
    H^0(N_{S/V}(mE)\big{\vert}_C) && \overset{res}\longleftarrow && 
    H^0(N_{S/V}(mE)) && \partial_E(\beta\big{\vert}_E) & \in & H^1(\mathcal O_E((m+1)E))
  \end{array}
  $$
  \caption{Relation between $\alpha$ and $\partial_E(\beta\big{\vert}_E)$}
  \label{fig:relation}
\end{figure}
We use the same strategy as \cite{Mukai-Nasu},
in which the proof is separated into 3 steps.
We follow the same steps and prove
$\overline{\ob_S(\alpha)}_{(m+1)}\ne 0$ 
in $H^1(C,N_{S/V}((m+1)E)\big{\vert}_C)$
instead of proving $\ob_S(\alpha)\ne 0$ in $H^1(C,N_{S/V}\big{\vert}_C)$.

Let $\alpha$ be a global section of $N_{C/V}$,
and $\pi_{C/S}(\alpha)\in H^0(C,N_{S/V}\big{\vert}_C)$ 
the exterior component of $\alpha$.
Suppose that the image 
$\overline{\pi_{C/S}(\alpha)}_{(m)}$ of $\pi_{C/S}(\alpha)$ 
in $H^0(C,N_{S/V}(mE)\big{\vert}_C)$ lifts to a section 
$\beta \in H^0(S,N_{S/V}(mE))\setminus H^0(S,N_{S/V}((m-1)E))$
for some integer $m \ge 1$,
i.e., an infinitesimal deformation with a pole along $E$.
We need the relation between the two $d$-maps 
$d_{C,S}$ and $d_S$ (cf.~Definition~\ref{dfn:d-map}),
allowing a pole along $E$.
The polar version of the diagram \eqref{diag:d-map} 
is the following {\it partially commutative} diagram:

\begin{equation}\label{diag:d-map with pole}
 \begin{array}{ccccc}
  \beta & \in & H^0(S,N_{S/V}(mE)) & \mapright{d_S} & 
   H^1(S,\mathcal O_S((m+1)E)) \\
  && & & \mapdown{|_C}\\
  && \mapdown{|_C} && H^1(C,\mathcal O_C((m+1)Z)) \\
  && && \mapdown{H^1(\iota)} \\
  && H^0(C,N_{S/V}(mE)\big{\vert}_C)& & 
   H^1(C,{N_{C/V}}^{\vee}\otimes N_{S/V}((m+1)E)\big{\vert}_C)\\
  && \cup &&\uparrow\\
  \gamma & \in & H^0(C,N_{S/V}\big{\vert}_C) & \mapright{d_{C,S}} & 
   H^1(C,{N_{C/V}}^{\vee}\otimes N_{S/V}\big{\vert}_C),
 \end{array}
\end{equation}
in which, the commutativity holds only for 
$\gamma \in H^0(C,N_{S/V}\big{\vert}_C)$ 
which has a lift $\beta \in H^0(S,N_{S/V}(mE))$.
More precisely, for such a pair $\gamma$ and $\beta$, we have
\begin{equation}\label{eq:d-map with pole}
\overline{d_{C,S} (\gamma)}_{(m+1)} = H^1(\iota)(d_S(\beta)\big{\vert}_C).
\end{equation}

\medskip \noindent 
{\bf Step 1} \qquad
$\overline{\ob_S(\alpha)}_{(m+1)}
=d_{S}(\beta)\big{\vert}_C \cup \pi_{C/S}(\alpha) 
\quad \text{in} \quad H^1(C,N_{S/V}((m+1)E)\big{\vert}_C)$.

\begin{proof}
By Lemma~\ref{lem:exterior}, we have
$\overline{\ob_S(\alpha)}_{(m+1)}
=\overline{d_{C,S}(\pi_{C/S}(\alpha))}_{(m+1)} \cup \alpha$.
Then it follows from \eqref{eq:d-map with pole} that
$\overline{d_{C,S}(\pi_{C/S}(\alpha))}_{(m+1)}=
H^1(\iota)(d_S(\beta)\big{\vert}_C)$.
By the commutative diagram
$$
\begin{array}{ccccc}
 H^1({N_{C/V}}^{\vee}\otimes N_{S/V}((m+1)E)) & \times 
  & H^0(N_{C/V}) & \overset{\cup}{\longrightarrow} 
  & H^1(N_{S/V}((m+1)E)\big{\vert}_C)\\
 \mapup{H^1(\iota)} && \mapdown{\pi_{C/S}} && \Vert \\
 H^1(\mathcal O_C((m+1)Z)) & \times & H^0(N_{S/V}\big{\vert}_C) 
  & \overset{\cup}{\longrightarrow} & H^1(N_{S/V}((m+1)E)\big{\vert}_C),
\end{array}
$$
we have the required equation.
\end{proof}

Next we relate $\ob_S(\alpha)$ with a cohomology class on $E$.
Let $\mathbf k_C$ and $\mathbf k_E$ be the extension classes 
defined by \eqref{ses:CE}.

\medskip \noindent
{\bf Step 2} \qquad
$\overline{\ob_S(\alpha)}_{(2m)} \cup \mathbf k_C 
=(d_S(\beta)\big{\vert}_E \cup \beta\big{\vert}_E) \cup \mathbf k_E 
\quad \text{in} \quad H^2(S,N_{S/V}(2mE-C))$.

\begin{proof}
We note that for every integers $i$, $n \ge 0$ and
for any coherent sheaf $\mathcal F$ on $S$, 
the map $H^i(S,\mathcal F)\rightarrow H^i(S,\mathcal F(nE))$,
$* \mapsto \overline *_{(n)}$,
and the cup product maps are compatible. For example,
the diagram
$$
\begin{CD}
  H^0(C,N_{S/V}\big{\vert}_C) @>{d_S(\beta)\cup}>> H^1(C,N_{S/V}((m+1)E)\big{\vert}_C)\\
  @VVV @VVV \\
  H^0(C,N_{S/V}((m-1)E)\big{\vert}_C) @>{d_S(\beta)\cup}>> H^1(C,N_{S/V}(2mE)\big{\vert}_C)
\end{CD}
$$
is commutative. Therefore, by Step 1 we have 
$$
\overline{\ob_S(\alpha)}_{(2m)}
=\overline{d_{S}(\beta)\big{\vert}_C \cup \pi_{C/S}(\alpha)}_{(m-1)}
=d_{S}(\beta)\big{\vert}_C \cup \overline{\pi_{C/S}(\alpha)}_{(m-1)}
$$
in $H^1(C,N_{S/V}(2mE)\big{\vert}_C)$.
Since there exists a commutative diagram
\begin{equation}\label{diag:res-coboundary1}
\begin{array}{ccccc}
 && \overline{\pi_{C/S}(\alpha)}_{(m-1)} && \overline{\ob_S(\alpha)}_{(2m)}\\ [-10pt]
  && \rotatebox{-90}{$\in$} && \rotatebox{-90}{$\in$} \\ [4pt]
  H^1(\mathcal O_C((m+1)Z)) & \times & H^0(N_{S/V}((m-1)E)\big{\vert}_C) 
  & \overset{\cup}{\longrightarrow} & H^1(N_{S/V}(2mE)\big{\vert}_C)\\
 \mapup{|_C} && \mapdown{\cup \, \mathbf k_C} && 
  \mapdown{\cup \, \mathbf k_C} \\
 H^1(\mathcal O_S((m+1)E)) & \times & H^1(N_{S/V}((m-1)E-C))
  & \overset{\cup}{\longrightarrow} & H^2(N_{S/V}(2mE-C)), \\
 \rotatebox{90}{$\in$} && && \\ [-4pt] 
  d_S(\beta) && && \\ 
\end{array}
\end{equation}
we have
$$
\overline{\ob_S(\alpha)}_{(2m)} \cup \mathbf k_C
=(d_{S}(\beta)\big{\vert}_C \cup 
\overline{\pi_{C/S}(\alpha)}_{(m-1)}) \cup \mathbf k_C
= d_{S}(\beta) \cup (\overline{\pi_{C/S}(\alpha)}_{(m-1)} \cup \mathbf k_C)
$$
in $H^2(S,N_{S/V}(2mE-C))$.
Then by Lemma~\ref{lem:restriction to C and E},
$\beta$ is contained in the subgroup
$H^0(S,\mathcal I_{Z/S} \otimes N_{S/V}(mE)) \subset H^0(S,N_{S/V}(mE))$,
and its restriction $\beta\big{\vert}_C$ to $C$ is a global section 
of the invertible sheaf
$$
\mathcal I_{Z/S} \otimes N_{S/V}(mE)\big{\vert}_C
\simeq \mathcal O_C(-Z)\otimes N_{S/V}(mE)\big{\vert}_C
\simeq N_{S/V}((m-1)E)\big{\vert}_C
$$
on $C$ and we have 
$\beta\big{\vert}_C = \overline{\pi_{C/S}(\alpha)}_{(m-1)}$
by assumption. Therefore we obtain
$$
d_{S}(\beta) \cup (\overline{\pi_{C/S}(\alpha)}_{(m-1)} \cup \mathbf k_C)
=d_{S}(\beta) \cup (\beta\big{\vert}_C \cup \mathbf k_C).
$$
Then \cite[Lemma~2.8]{Mukai-Nasu} shows that we have
$\beta\big{\vert}_C \cup \mathbf k_C
=\beta\big{\vert}_E \cup \mathbf k_E$
in $H^1(S,N_{S/V}((m-1)E-C))$.
Hence we have
$$
d_{S}(\beta) \cup (\beta\big{\vert}_C \cup \mathbf k_C)
=d_{S}(\beta) \cup (\beta\big{\vert}_E \cup \mathbf k_E)
=(d_{S}(\beta)\big{\vert}_E \cup \beta\big{\vert}_E) \cup \mathbf k_E,
$$
where the last equality follows from the commutative diagram
\begin{equation}\label{diag:res-coboundary2}
  \begin{array}{ccccc}
    && \beta\big{\vert}_E && \\ [-10pt]
    && \rotatebox{-90}{$\in$} && \\ [4pt]
    H^1(\mathcal O_E((m+1)E)) & \times & H^0(N_{S/V}(mE-C)\big{\vert}_E)
    & \overset{\cup}{\longrightarrow} & H^1(N_{S/V}((2m+1)E-C)\big{\vert}_E) \\
    \mapup{|_E} && \mapdown{\cup \, \mathbf k_E} && 
    \mapdown{\cup \, \mathbf k_E} \\
    H^1(\mathcal O_S((m+1)E)) & \times & H^1(N_{S/V}((m-1)E-C))
    & \overset{\cup}{\longrightarrow} & H^2(N_{S/V}(2mE-C)), \\
    \rotatebox{90}{$\in$} && && \\ [-4pt] 
    d_S(\beta) && && \\ 
  \end{array}
\end{equation}
similar to \eqref{diag:res-coboundary1}.
Thus we obtain the equation required.
\end{proof}

\medskip \noindent
{\bf Step 3} \quad 
Let $\partial_E$ be the coboundary map of \eqref{ses:normal bundle of E}.
Then by Proposition~\ref{prop:key}~(2),
we have 
$d_{S}(\beta)\big{\vert}_E\cup \beta\big{\vert}_E
=m\partial_E (\beta\big{\vert}_E)\cup \beta\big{\vert}_E$,
which is nonzero by the assumption (b).
Consider the coboundary map
$$
\cup\, \mathbf k_E: H^1(E,N_{S/V}((2m+1)E-C)\big{\vert}_E) \longrightarrow
H^2(S,N_{S/V}(2mE-C)),
$$
which appears in \eqref{diag:res-coboundary2}.
By the Serre duality, it is dual to the restriction map
$$
H^0(S,C+K_V\big\vert_S -2mE) \overset{|_E}{\longrightarrow} H^0(E,(C+K_V-2mE)\big\vert_E),
$$
which is surjective by the assumption (a).
Hence the coboundary map $\cup\, \mathbf k_E$ is injective.
Therefore we obtain
$d_{S}(\beta)\big{\vert}_E \cup \beta\big{\vert}_E \cup \mathbf k_E\ne 0$
and hence by Step 2 we conclude that
$\overline{\ob_S(\alpha)}_{(2m)}\ne 0$ in $H^1(C,N_{S/V}(2mE)\big{\vert}_C)$,
and hence we have finished the proof of Theorem~\ref{thm:main1}.
\qed

\medskip

Let 
$\pi_{E/S}(mE): H^0(E,N_{E/V}(mE)) \rightarrow H^0(E,N_{S/V}(mE)\big{\vert}_E)$ 
be the map induced by \eqref{ses:normal bundle of E}.
If this map is not surjective, then
the sections $\gamma$ in its image span a proper linear subsystem
$$
\Lambda':=\left\{
  \div(\gamma) \bigm|
  \gamma \in 
  \im \pi_{E/S}(mE)
\right\}
$$
of the complete linear system $\Lambda:=|N_{S/V}(mE)\big{\vert}_E|$ on $E$,
where $\div (\gamma)$ denotes the divisor of zeros for $\gamma$.
The condition (b) in Theorem~\ref{thm:main1} can be replaced 
with the following conditions (b1), (b2) and (b3) in the next corollary,
which is more accessible in many situations.

\begin{cor}
  \label{cor:obstruction}
  Let $C \subset S\subset V$, $E$, $\alpha$, $\beta$, $\Delta$
  be as in Theorem~\ref{thm:main1}.
  Suppose that $\beta\big{\vert}_E \ne 0$.
  If the following conditions are satisfied, then the 
  exterior component $\ob_S(\alpha)$ of $\ob(\alpha)$ is nonzero.
  \begin{enumerate}
    \renewcommand{\theenumi}{{\alph{enumi}}}
    \item The restriction map 
    $
    H^0(S,\Delta) \overset{\vert_E}\longrightarrow 
    H^0(E,\Delta\big{\vert}_E)
    $
    is surjective,
    \item[(b1)] $m$ is not divisible by the characteristic $p$
    of the ground field $k$, 
    \item[(b2)] $E$ is irreducible curve of arithmetic genus $g(E)$ 
    and $(\Delta.E)=2g(E)-2-(m+1)E^2$.
   \item[(b3)] $Z:=C\cap E$ is not a member of $\Lambda'$.
  \end{enumerate} 
\end{cor}
\begin{proof}
It suffices to prove that the condition (b) of Theorem~\ref{thm:main1}
follows from the conditions (b1), (b2) and (b3) of this corollary.
Since we have
$N_{S/V}(mE-C) \simeq -K_V\big{\vert}_S+K_S+mE-C = K_S-\Delta-mE$,
there exists an isomorphism
\begin{equation}
  \label{isom:triviality}
  N_{S/V}(mE-C)\big{\vert}_E 
  \simeq \mathcal O_E(K_E-\Delta-(m+1)E) 
\end{equation}
of invertible sheaves on $E$, whose degree is zero by (b2).
By Lemma~\ref{lem:restriction to C and E},
$\beta\big{\vert}_E$ is a nonzero global section of 
$N_{S/V}(mE-C)\big{\vert}_E$.
Since $E$ is irreducible,
the invertible sheaves in \eqref{isom:triviality} are trivial.
Hence as a section of $N_{S/V}(mE)\big{\vert}_E$,
the divisor $\div(\beta\big{\vert}_E)$ of zeros 
associated to $\beta\big{\vert}_E$ coincides with $Z$.
It follows from (b3) that $\partial_E(\beta\big{\vert}_E)\ne 0$
in $H^1(E,\mathcal O_E((m+1)E))$.
Since $\beta\big{\vert}_E$ is
a nonzero section of the trivial sheaf $N_{S/V}(mE-C)\big{\vert}_E$,
the cup product
$m\partial_E(\beta\big{\vert}_E)\cup \beta\big{\vert}_E$
is nonzero.
\end{proof}

We finish this section by giving a refinement of Theorem~\ref{thm:main1}.
Let $E_i$ ($1 \le i \le k$) be irreducible curves on $S$,
which are mutually disjoint and $C \ne E_i$ for any $i$.
We assume that for any two effective divisors $D,D'$ on $S$
with supports on $\bigcup_{i=1}^k E_i$, if $D \le D'$ then
the map $H^1(S,\mathcal O_S(D)) \rightarrow H^1(S,\mathcal O_S(D'))$
is injective.

\begin{thm}\label{thm:refinement}
  Let $E=\sum_{i=1}^k m_i E_i$ be a divisor on $S$ with $m_i \ge 1$.
  Let $\tilde C$ or $\alpha \in H^0(C,N_{C/V})$
  be a first order deformation of $C$.
  Suppose that $H^1(S,N_{S/V})=0$ and 
  the image of the exterior component
  $\pi_{C/S}(\alpha)$ in $H^0(C,N_{S/V}(E)\big{\vert}_C)$
  lifts to an infinitesimal deformation 
  $\beta \in H^0(S,N_{S/V}(E))$ with poles along $E$.
  If the following conditions are satisfied,
  then the exterior component 
  $\ob_S(\alpha)$ of $\ob(\alpha)$ is nonzero:
  \begin{enumerate}
    \renewcommand{\theenumi}{{\alph{enumi}}}
    \item \label{item:surjective}
    Let $\Delta:=C+K_V\big{\vert}_S-2E$ in $\Pic S$
    and let $E_{\red}:=\sum_{i=1}^k E_i$,
    i.e., the reduced part of $E$.
    Then the restriction map 
    $$
    H^0(S,\Delta) \overset{\vert_{E_{\red}}}\longrightarrow 
    H^0(E_{\red},\Delta\big{\vert}_{E_{\red}})
    $$
    to $E_{\red}$ is surjective, and 
    \item \label{item:nonzero}
    Let $\beta\big{\vert}_{E_i} \in H^0(E_i,N_{S/V}(m_iE_i)\big{\vert}_{E_i})$
    be the principal part of $\beta$ along $E_i$,
    and let $\partial_{E_i}$ be the coboundary map 
    of the exact sequence \eqref{ses:normal bundle of E} 
    for $E=E_i$ and $m=m_i$.
    Then there exists an integer $1 \le i \le k$ such that
    $$
    m_i\partial_{E_i}(\beta\big{\vert}_{E_i}) \cup
    \beta\big{\vert}_{E_i} \ne 0 
    \qquad
    \mbox{in}
    \qquad
    H^1(E_i,N_{S/V}((2m_i+1){E_i}-C)\big{\vert}_{E_i}).
    $$
  \end{enumerate}
\end{thm}
\begin{proof}
The proof is similar to that of Theorem~\ref{thm:main1}.
We follow the same steps in the proof.
Let $d_{S^\circ}: H^0(S^{\circ},N_{S^{\circ}/V^{\circ}}) 
\rightarrow H^1(S^{\circ},\mathcal O_{S^\circ})$
be the $d$-map for $S^{\circ} \subset V^{\circ}$.
Then the image $d_S(\beta)$ of $\beta \in H^0(S,N_{S/V}(E))$
is contained in $H^1(S,\mathcal O_S(E^+))
\subset H^1(S^{\circ},\mathcal O_{S^{\circ}})$
by Proposition~\ref{prop:key2},
where $E^+:=E+\sum_{i=1}^kE_i=\sum_{i=1}^k(m_i+1)E_i$.
There exists a partially commutative diagram similar to
\eqref{diag:d-map with pole}, which connects 
the two polar $d$-maps $d_{C,S}$ and $d_S$ with poles along $E$.
By using this diagram, we find
$$
r(d_{C,S}(\pi_{C/S}(\alpha)),E^{+})=
H^1(\iota)(d_S(\beta)\big{\vert}_C)
\qquad
\mbox{in}
\qquad
H^1(C,{N_{C/V}}^{\vee}\otimes N_{S/V}(E^+)\big{\vert}_C).
$$
Here and later, 
given a coherent sheaf $\mathcal F$ on $S$ (resp. $C$),
an effective divisor $D\ge 0$ on $S$,
an integer $i\ge 0$
and a cohomology class $*$ in $H^i(S,\mathcal F)$,
we denote by $r(*,D)$ the image of
$*$ by the natural map
$$
H^i(S,\mathcal F) \rightarrow H^i(S,\mathcal F(D))
$$
and use a similar notation for $C$ also.
We will show that 
$$
r(\ob_S(\alpha),E^+)\ne 0
\qquad
\mbox{in}
\qquad
H^1(C,N_{S/V}(E^+)\big{\vert}_C).
$$
Using the argument in Step 1 before, we get
$r(\ob_S(\alpha),E^+)=d_S(\beta)\big{\vert}_C \cup
\pi_{C/S}(\alpha)$,
where the cup product is taken by the map
$$
H^1(C,\mathcal O_C(E^+))\times
H^0(C,N_{S/V}\big{\vert}_C)
\overset{\cup}\longrightarrow 
H^1(C,N_{S/V}(E^+)\big{\vert}_C).
$$
We consider the cup product
$$
r(\ob_S(\alpha),E^+)\cup \mathbf k_C
\qquad
\mbox{in}
\qquad
H^2(S,N_{S/V}(E^+-C))
$$
of $\overline{\ob_S(\alpha)}$
with $\mathbf k_C \in \Ext^1_S(\mathcal O_C,\mathcal O_S(-C))$
and its reduced image in 
$$
r(r(\ob_S(\alpha),E^+),E^-)\cup \mathbf k_C
\qquad
\mbox{in}
\qquad
\in H^2(S,N_{S/V}(2E-C)),
$$
where $E^-$ is a divisor on $S$ defined by
$E^-:=E-\sum_{i=1}^k E_i=\sum_{i=1}^k (m_i-1)E_i$.
Then there exists a commutative diagram
$$
\begin{array}{ccccc}
  H^1(C,\mathcal O_C(E^+)) & \times & H^0(C,N_{S/V}\big{\vert}_C)
  & \overset{\cup}{\longrightarrow} & H^1(C,N_{S/V}(E^+)\big{\vert}_C)\\
  \Vert && \mapdown{r} && \mapdown{r} \\
  H^1(C,\mathcal O_C(E^+)) & \times & H^0(C,N_{S/V}(E^-)\big{\vert}_C) 
  & \overset{\cup}{\longrightarrow} & H^1(C,N_{S/V}(2E)\big{\vert}_C)\\
 \mapup{|_C} && \mapdown{\cup \, \mathbf k_C} && 
  \mapdown{\cup \, \mathbf k_C} \\
 H^1(S,\mathcal O_S(E^+)) & \times & H^1(S,N_{S/V}(E^--C))
  & \overset{\cup}{\longrightarrow} & H^2(S,N_{S/V}(2E-C)),
\end{array}
$$
since $E^++E^-=2E$ by definition.
By using this diagram, we compute the cup product
$r(r(\ob_S(\alpha),E^+)\cup \mathbf k_C,E^-)$ 
as
\begin{align*}
  r(r(\ob_S(\alpha),E^+)\cup \mathbf k_C,E^-)
  &=r(d_S(\beta)\big{\vert}_C \cup \pi_{C/S}(\alpha)\cup \mathbf k_C,E^-) \\
  &= (d_S(\beta)\big{\vert}_C \cup r(\pi_{C/S}(\alpha),E^-)) \cup \mathbf k_C \\
  &= d_S(\beta)\cup (r(\beta\big{\vert}_C,E^-) \cup \mathbf k_C).
\end{align*}
Let $E_{\red}=\sum_{i=1}^k E_i$ be the reduced part of $E$
as in the statement.
Then for $E_{\red}$ there exists another commutative diagram
$$
\begin{array}{ccccc}
 H^1(S,\mathcal O_S(E^+)) & \times & H^1(S,N_{S/V}(E^--C))
  & \overset{\cup}{\longrightarrow} & H^2(S,N_{S/V}(2E-C)) \\
 \mapdown{|_{E_{\red}}} && \mapup{\cup \, \mathbf k_{E_{\red}}} && 
  \mapup{\cup \, \mathbf k_{E_{\red}}} \\
  H^1(E_{\red},\mathcal O_{E_{\red}}(E^+)) & \times &
  H^0(E_{\red},N_{S/V}(E-C)\big{\vert}_{E_{\red}})
  & \overset{\cup}{\longrightarrow} & H^1(E_{\red},N_{S/V}(E+E^+-C)\big{\vert}_{E_{\red}}).
\end{array}
$$
Since $E_i$ are mutually disjoint, 
here each cohomology group on $E_{\red}$ is isomorphic to a direct sum
of cohomology groups on $E_i$ and we obtain
\begin{align*}
  H^1(E_{\red},\mathcal O_{E_{\red}}(E^+)) &\simeq 
  \bigoplus_{i=1}^k H^1(E_i,\mathcal O_{E_i}((m_i+1)E_i)) \\
  H^0(E_{\red},N_{S/V}(E-C)\big{\vert}_{E_{\red}}) & \simeq 
  \bigoplus_{i=1}^k H^0(E_i,N_{S/V}(m_iE_i-C)) \\
  H^1(E_{\red},N_{S/V}(E+E^+-C)\big{\vert}_{E_{\red}}) &\simeq
  \bigoplus_{i=1}^k H^1(E_i,N_{S/V}((2m_i+1)E_i-C)).
\end{align*}
Then since the ideal of $C\cap E_{\red}$ in $S$
has a Koszul resulution
$0 \rightarrow \mathcal O_S(-C-E_{\red})
\rightarrow \mathcal O_S(-C)\oplus \mathcal O(-E_{\red})
\rightarrow \mathcal I_{C\cap E_{\red}/S}
\rightarrow 0$,
by \cite[Lemma~2.8]{Mukai-Nasu}, we have
$r(\beta\big{\vert}_C,E^-) \cup \mathbf k_C
=\beta\big{\vert}_{E_{\red}} \cup \mathbf k_{E_{\red}}$
in $H^1(S,N_{S/V}(E^--C))$.
Then we obtain
\begin{equation}
  \label{eqn:reduced cup product}
  d_S(\beta)\cup ((r(\beta\big{\vert}_C,E^-) \cup \mathbf k_C) 
  = (d_S(\beta)\big{\vert}_{E_{\red}} \cup \beta\big{\vert}_{E_{\red}})
  \cup \mathbf k_{E_{\red}}. 
\end{equation}
Then by Proposition~\ref{prop:key2}(2) and the assumption \eqref{item:nonzero},
we have
$$
d_S(\beta)\big{\vert}_{E_{\red}} \cup \beta\big{\vert}_{E_{\red}}
=\oplus_{i=1}^k d_S(\beta)\big{\vert}_{E_i}\cup \beta\big{\vert}_{E_i}
=\oplus_{i=1}^k m_i\partial_{E_i}(\beta\big{\vert}_{E_i})\cup \beta\big{\vert}_{E_i}\ne 0.
$$
It follows from the assumption \eqref{item:surjective} that the cup product map
$$
H^1(E_{\red},N_{S/V}(E+E^+-C)\big{\vert}_{E_{\red}}) \overset{\cup \mathbf k_{E_{\red}}}{\longrightarrow}
H^2(S,N_{S/V}(2E-C))
$$
with $\mathbf k_{E_{\red}} \in \Ext^1_S(\mathcal O_{E_{\red}},\mathcal O_S(-E_{\red}))$
is injective.
Therefore
$(d_S(\beta)\big{\vert}_{E_{\red}} \cup \beta\big{\vert}_{E_{\red}})
  \cup \mathbf k_{E_{\red}}$ is nonzero
and thus we have completed the proof.
\end{proof}

\begin{rmk}
  \label{rmk:newly added!}
  Note that the cup product map $\cup \mathbf k_{E_{\red}}$ is 
  the direct sum $\oplus_{i=1}^k \cup \mathbf k_{E_i}$ of 
  $\cup \mathbf k_{E_i}$ with each $\mathbf k_{E_i}$ ($1 \le i \le k$)
  because $E_i$ are mutually disjoint.
  The last cup product in the equation
  \eqref{eqn:reduced cup product} is expressed as
  a sum of the coboundary image of 
  $m_i\partial_{E_i}(\beta\big{\vert}_{E_i})\cup \beta\big{\vert}_{E_i}$
  on each $E_i$ as follows:
  $$
  (d_S(\beta)\big{\vert}_{E_{\red}} \cup \beta\big{\vert}_{E_{\red}})
  \cup \mathbf k_{E_{\red}}
  =\sum_{i=1}^k
  m_i\partial_{E_i}(\beta\big{\vert}_{E_i})\cup \beta\big{\vert}_{E_i}
  \cup \mathbf k_{E_i},
  $$
  which is an element in $H^1(S,N_{S/V}(2E-C))$.
\end{rmk}

\section{Obstructions to deforming curves lying on a $K3$ surface} 
\label{sect:k3}

In this section, we prove Theorem~\ref{thm:main2}
and Corollary~\ref{cor:main3}.
In this and later sections, we assume that $\car k=0$.
Let $C \subset S \subset V$ be as in the theorem.
Then by Lemma~\ref{lem:flag of k3fano},
the Hilbert-flag scheme $\HF V$ is nonsingular at $(C,S)$,
and moreover, $A^2(C,S)=0$.
Put $D:=C+K_V\big{\vert}_S$, a divisor on $S$.
Then by the same lemma together with the Serre duality,
we have $H^i(S,N_{S/V}(-C)) \simeq  H^i(S,-D)\simeq H^{2-i}(S,D)^{\vee}$
for any integer $i$.

\paragraph{{\bf Proof of Theorem~\ref{thm:main2}}.}\quad
(1)\quad 
It is known that if there exist
no $(-2)$-curves and no elliptic curves on a smooth $K3$ surface $X$,
then every nonzero effective divisor on $X$ is ample.
(Then the effective cone $\eff(X)$ and the 
  ample cone $\amp(X)$ coincides.)
Therefore we have $H^1(S,D)=0$ by assumption.
Then the restriction map
$\rho: H^0(S,N_{S/V})\rightarrow H^0(C,N_{S/V}\big{\vert}_C)$
is surjective. Then by Lemma~\ref{lem:flag to hilb}(1), 
$\Hilb V$ is nonsingular at $[C]$.

(2)\quad
We show that the tangent map $p_1$ of 
$pr_1: \HF V \rightarrow \Hilb V$ at $(C,S)$ 
is not surjective and its cokernel is of dimension $1$.
Since $D \ge 0$ and $D\ne 0$, we have $H^0(S,-D)=0$.
Therefore by \eqref{seq:flag to hilb}, 
there exists an exact sequence 
\begin{equation}
  \label{ses:flag to hilb}
  0 \longrightarrow A^1(C,S) \overset{p_1}
  \longrightarrow H^0(C,N_{C/V}) 
  \longrightarrow H^1(S,-D)
  \longrightarrow 0. 
\end{equation}

\begin{claim}
  \label{claim:non-S-linear}
  $H^1(S,-D)\simeq k$ and $H^1(S,-D+E)=0$
\end{claim}

\paragraph{\bf Proof of Claim~\ref{claim:non-S-linear}.}\quad 
Since $D.E=-2$ and $E^2=-2$,
there exists an exact sequence
\begin{equation}
\label{ses:(-2)-P^1}
0 \longrightarrow \mathcal O_S(D-(l+1)E)
\longrightarrow \mathcal O_S(D-lE)
\longrightarrow \mathcal O_{\mathbb P^1}(2l-2)
\longrightarrow 0
\end{equation}
for every integer $l$.
Since $H^1(\mathbb P^1,\mathcal O_{\mathbb P^1}(2l-2))=0$ for $l \ge 1$
and $H^1(S,D-3E)=0$, it follows from this exact sequence that
$H^1(S,D-lE)=0$ for $l=1,2$.
We prove $H^2(S,D-E)=0$. In fact, if $H^2(S,D-E)\ne 0$, then
by the Serre duality, $-D+E$ is effective 
and we have $(-D+E)^2=D^2+2 \ge 2 >0$.
Then it follows from the signature theorem 
(cf.~\cite[IV.2, Thm.~2.14 and VIII.1]{BHPV})
that $0 \le (D.-D+E) =-D^2-2<0$ and hence we get a contradiction.
Therefore by \eqref{ses:(-2)-P^1}, we have
$H^1(S,D) \simeq H^1(\mathbb P^1,\mathcal O_{\mathbb P^1}(-2)) \simeq k$.
By the Serre duality, we have proved the claim. 

\medskip

Let $\alpha$ be a global section of $N_{C/V}$. 
It suffices to prove the next claim.
\begin{claim}
  \label{claim:obstruction}
  $\ob_S(\alpha)\ne 0$ if $\alpha \not \in \im p_1$.
\end{claim}

\paragraph{\bf Proof of Claim~\ref{claim:obstruction}.}\quad
Let $\pi_{C/S}(\alpha)\in H^0(C,N_{S/V}\big{\vert}_C)$ 
be the exterior component of $\alpha$.
There exists a commutative diagram
$$
\begin{CD}
  0 @>>> A^1(C,S) @>{p_1}>> H^0(C,N_{C/V}) @>>> H^1(S,-D) @>>> 0 \\
  @. @V{p_2}VV @V{\pi_{C/S}}VV \Vert \\
  0 @>>> H^0(S,N_{S/V}) @>{\rho}>> H^0(C,N_{S/V}\big{\vert}_C) @>{\cup \mathbf k_C}>> H^1(S,-D) @>>> 0,
\end{CD}
$$
where $\mathbf k_C$ is the extension class of \eqref{ses:CE}.
By this diagram, we see that
$\pi_{C/S}(\alpha)$ is not contained in $\im \rho$, and hence
$\pi_{C/S}(\alpha)\cup \mathbf k_C\ne 0$ in $H^1(S,-D)$.
On the other hand, since $H^1(S,-D+E)=0$,
we have $\overline{\pi_{C/S}(\alpha)}_{(1)}\cup \mathbf k_C=0$
and $\pi_{C/S}(\alpha)$ lifts to an infinitesimal deformation 
$\beta \in H^0(S,N_{S/V}(E))$ with a pole along $E$
by Lemma~\ref{lem:restriction to C and E}(1).
Then by (2) of the same lemma,
the principal part $\beta\big{\vert}_E$ of $\beta$
is a nonzero global section of $N_{S/V}(E)\big{\vert}_E$,
and its divisor of zeros contains $Z:=C\cap E$
(i.e. $Z \subset \div (\beta\big{\vert}_E)$).

Now we verify that the four conditions (a), (b1), (b2) and (b3) of 
Corollary~\ref{cor:obstruction} are satisfied.
Put $\Delta=C+K_V\big{\vert}_S-2E$, a divisor on $S$.
Since $H^1(S,\Delta-E)=H^1(S,D-3E)=0$, 
(a) follows from the exact sequence
$0 \rightarrow \mathcal O_S(\Delta-E)
\rightarrow \mathcal O_S(\Delta)
\rightarrow \mathcal O_E(\Delta)
\rightarrow 0.
$
(b1) is clear.
Since $E$ is a $(-2)$-curve, we compute that 
$\Delta.E=(D-2E.E)=2=2g(E)-2-2E^2$, and hence (b2) follows. 
Then by \eqref{isom:triviality}, this implies that 
$N_{S/V}(E-C)\big{\vert}_E$ is trivial and hence
we have $Z =\div (\beta\big{\vert}_E)$.
Finally, for (b3), we show that $H^1(S,C-E)=0$.
In fact, we have $C-E=D-E-K_V\big{\vert}_S$.
Since $H^1(S,D-E)=0$, 
we have $(D-E.E')\ge -1$ for any $(-2)$-curve $E'$ on $S$
by Lemma~\ref{lem:K3}(2), which implies that $C-E$ is nef
because $-K_V\big{\vert}_S$ is ample. Then $C-E$ is big by
$$
(C-E)^2=(D-E)^2+2(D-E.-K_V\big{\vert}_S)+(-K_V\big{\vert}_S)^2
>(D-E)^2>0.
$$
Therefore $H^1(S,C-E)=0$ by the Kodaira-Ramanujam vanishing theorem.
Then the rational map
$$
|C| \dashrightarrow |\mathcal O_E(C)|, \qquad C' \longmapsto Z=C'\cap E
$$
is dominant. By assumption, 
$\Lambda'=\left\{\div(\gamma) \bigm|
\gamma \in \im \pi_{E/S}(E)
\right\}$ 
is a proper linear subsystem of $\Lambda=|N_{S/V}(E)\big{\vert}_E|$.
Therefore, if necessary, by replacing $C$ with a general member $C'$
of $|C|$, we may assume that
$Z=C\cap E$ is not contained in $\Lambda'$.
In fact, by the upper semicontinuity, if a general member $C'\in |C|$
is obstructed, then so is $C$. Hence (b3) follows.
By Corollary~\ref{cor:obstruction}, we have proved the claim.

(3)\quad
The proof is very similar to that of (2).
Suppose that $D\sim mF$ for $m \ge 2$ and an elliptic curve $F$.
Then by Lemma~\ref{lem:K3}(2), we have $H^1(S,-D)\simeq k^{m-1}$.
Thus the cokernel of the tangent map $p_1$ of $pr_1$ is nonzero.
Since $H^1(S,-D+F)\simeq k^{m-2}$, the kernel of 
the natural map $H^1(S,-D)\rightarrow H^1(S,-D+F)$ 
is of dimension at least one.
Hence by \eqref{ses:flag to hilb}, 
there exists a global section $\alpha$ of $N_{C/V}$
whose exterior component $\pi_{C/S}(\alpha)$ satisfies
$\pi_{C/S}(\alpha)\cup \mathbf k_C \ne 0$ in $H^1(S,-D)$, while
$\overline{\pi_{C/S}(\alpha)}_{(1)}\cup \mathbf k_C=0$ in $H^1(S,-D+F)$.
We fix such a global section $\alpha$
and prove that $\ob_{S}(\alpha)\ne 0$.
Again by Lemma~\ref{lem:restriction to C and E},
there exists an infinitesimal deformation 
$\beta \in H^0(S,N_{S/V}(F))$ with a pole along $F$
such that
$\beta\big{\vert}_C = \overline{\pi_{C/S}(\alpha)}_{(1)}$ and
$\beta\big{\vert}_F$ 
is a nonzero global section of $N_{S/V}(F)\big{\vert}_F$.
Thus it suffices to verify the conditions (a), (b1), (b2) and (b3) of
Corollary~\ref{cor:obstruction}.
(a) follows from $\Delta=(m-2)F$ and $\mathcal O_F(F)\simeq \mathcal O_F$.
(b1) is clear.
(b2) follows from $\Delta.F=2g(F)-2-2F^2=0$.
Since $C-F=-K_V\big{\vert}_S+(m-1)F$ is ample, we have $H^1(S,C-F)=0$
by the Kodaira vanishing theorem. Hence (b3) follows.
The rest of the proof is same as that of (2).
\qed

\paragraph{{\bf Proof of Corollary~\ref{cor:main3}}.}\quad
Let $W_{C,S}$ be the $S$-maximal family of curves containing $C$.
If $H^1(S,D)=0$, then by Lemma~\ref{lem:flag to hilb}(1), 
$pr_1$ is smooth.
Since a smooth morphism is flat and a flat morphism maps a 
generic point onto a generic point, we have the conclusion of (a)
(cf.~\cite[Corollary 1.3.5]{Kleppe81}).
Suppose that $h^1(S,D)=1$.
Then it follows from \eqref{ses:flag to hilb} that
$$
\dim W_{C,S} 
\le \dim_{[C]} \Hilb^{sc} V
\le h^0(C,N_{C/V})
=\dim W_{C,S}+1.
$$
Since $C$ is obstructed by the theorem,
we have $\dim_{[C]} \Hilb^{sc} V=\dim W_{C,S}$.
Therefore (a) follows.
Since $C$ is a generic member of $W_{C,S}$, we obtain (b).
If $H^0(S,-D)=0$, then by Lemma~\ref{lem:flag to hilb}(2)
and Lemma~\ref{lem:flag of k3fano},
we obtain $\dim_{[C]} \Hilb^{sc} V=
\dim W_{C,S} = (-K_V\big{\vert}_S)^2/2+g+1$.
Thus we have completed the proof.
\qed

\section{Non-reduced components of the Hilbert scheme}
\label{sect:non-reduced}

In this section, as an application, 
we study the deformations of curves 
lying on a smooth quartic surface $S$ in $\mathbb P^3$, 
or a smooth hyperplane section $S$
of a smooth quartic $3$-fold $V_4 \subset \mathbb P^3$ (assuming $\car k=0$).
We give some examples of generically non-reduced components
of the Hilbert schemes $\Hilb^{sc} \mathbb P^3$
and $\Hilb^{sc} V_4$ 
(cf.~Examples~\ref{ex:non-reduced V_4} and \ref{ex:non-reduced P^3}).
As is well known, $S$ is a $K3$ surface.
Here we consider $S$ 
(i) of Picard number two ($\rho(S)=2$),
(ii) containing a \underline{smooth} curve $E$
not a complete intersection in $S$, 
and such that 
(iii) $\Pic S$ is generated by $E$ and the class $\mathbf h$ 
of hyperplane sections of $S$
(i.e., $\Pic S \simeq \mathbb Z \mathbf h \oplus \mathbb Z E$).
Kleppe and Ottem \cite{Kleppe-Ottem} have studied
the deformations of space curves $C \subset \mathbb P^3$
lying on such a quartic surface $S$
by assuming that $E$ is a line, or a conic.
They have also produced examples of
generically non-reduced components of $\Hilb^{sc} \mathbb P^3$
by a different method (cf.~Remark~\ref{rmk:Kleppe-Ottem}).

\subsection{Mori cone of quartic surfaces}
\label{subsec:mori cone of quartics}

Let $S \subset \mathbb P^3$ be a smooth quartic surface.
If $S$ is general, then we have $\rho(S)=1$
and $\Pic S$ is generated by 
$\mathbf h=-\frac 14 K_{\mathbb P^3}\big{\vert}_S$.
Then every curve $C$ on $S$
is a complete intersection of $S$ with some other surface in $\mathbb P^3$, 
and hence $C$ is arithmetically Cohen-Macaulay.
Thus we see that $C$ is unobstructed,
thanks to a result of Ellingsrud~\cite{Ellingsrud}.

First we consider a quartic surface $S$ containing 
a (smooth) rational curve $E$.
Then by its genus, $E$ is not a complete intersection in $S$.
It follows from Lemma~\ref{lem:mori} that for every integer $e\ge 1$, 
there exists a smooth quartic surface
$S \subset \mathbb P^3$ containing a rational curve $E$ of degree $e$,
and such that $\Pic S \simeq \mathbb Z \mathbf h \oplus \mathbb Z E$.
Let $(S,E)$ be such a pair of 
a surface $S$
and a curve $E\simeq \mathbb P^1$. 
Then every divisor $D$ on $S$ is linearly equivalent to
$x\mathbf h-yE$ for some $x,y \in \mathbb Z$.
Since we have $\mathbf h^2=4, \mathbf h.E=e$ and $E^2=-2$,
we compute the self intersection number of $D$ on $S$
as $D^2=4x^2-2exy-2y^2$.
Recall that for a projective surface $X$,
the effective cone 
$\eff(X)$ of $X$ is defined as
$\eff(X)=\left\{
  \sum_{i=1}^n a_i[C_i] \bigm| 
  \mbox{$C_i$ is an irreducible curve on $X$ and $a_i \in \mathbb R_{\ge 0}$}
\right\}$
and the Mori cone $\mori(X)$ is defined as the closure of $\eff(X)$
in $\ns(X)_{\mathbb R}$.
Kov\'acs~\cite{Kovacs} proved that for every $K3$ surface $X$ 
with $\rho(X)=2$,
$\mori(X)$ has two extremal rays, which can be 
generated by the classes of two $(-2)$-curves, one $(-2)$-curve and
an elliptic curve, two elliptic curves, or two non-effective classes
$x_1,x_2$ with $x_i^2=0$.
Applying Kov\'acs's result, we have the following lemma.

\begin{lem}
  \label{lem:mori cone1}
  Let $S$ be a smooth quartic surface, 
  $E$ a smooth rational curve of degree $e \ge 2$ on $S$ such that
  $\Pic S=\mathbb Z\mathbf h\oplus \mathbb ZE$,
  $D$ a divisor on $S$.
  Then 
  \begin{enumerate}
    \item If $D$ is not linearly equivalent to $0$, then $D^2\ne 0$.
    \item There exists a (unique) $(-2)$-curve $E'$ on $S$ such that
    $\mori(S)=\mathbb R_{\ge 0}[E] +\mathbb R_{\ge 0}[E']$.
    \item $D$ is nef
    if and only if $D.E\ge 0$ and $D.E'\ge 0$.
    \item Suppose that $D\ge 0$ and $D \ne 0$. Then
    a general member $C$ of $|D|$ is a smooth connected curve
    if and only if
    $D$ is nef, $D=E$, or $D=E'$.
    \item If (i) $D\ge 0$ and $D$ is nef or 
    (ii) $D=E,E'$, then $H^1(S,D)=0$.
  \end{enumerate}
\end{lem}

\begin{proof}
Since $E$ spans one of the two extremal rays, for proving (1) and (2),
it suffices to prove that there exist no elliptic curves on $S$.
Suppose that there exists a nonzero divisor $D\sim x\mathbf h-yE$ on $S$
with $D^2=-2(y^2+exy-2x^2)=0$.
Then the discriminant $d=(ex)^2-4(-2x^2)=(8+e^2)x^2$ of this
quadratic equation (with a variable $y$)
equals a power of an integer.
Since $D\not\sim 0$ in $\Pic S$, we have $8+e^2=k^2$ for some integer $k \ge 1$.
By solving this equation, we have $(k,e)=(3,1)$, or $(9/2,7/2)$,
which are both impossible by assumption.
Since the nef cone and the Mori cone are dual to each other, 
we have (3).
If $C \in |D|$ is irreducible and $C \ne E,E'$, then $C$ is nef.
Conversely, if $D$ is nef, then
$|D|$ is base point free by Lemma~\ref{lem:K3}(1),
and a general member $C \in |D|$
is a smooth connected curve by Bertini's theorem.
Hence (4) follows. (5) is clear if $D\sim 0$, $E$, or $E'$.
Otherwise it follows from (1) and the Kodaira-Ramanujam vanishing theorem.
\end{proof}

\begin{rmk}
  \begin{enumerate}
    \item If $e=1$, then $\mori(S)$ is generated by the classes
    of a line $E$ and a smooth elliptic curve
    $F\sim \mathbf h -E$ of degree $3$,
    contained in a plane in $\mathbb P^3$
    (cf.~\cite[Proposition~5.1]{Kleppe-Ottem}).
    \item Once $e\ge 2$ is given, the class $x\mathbf h-yE$ of 
    the curve $E'\simeq \mathbb P^1$ spanning the other ray
    of $\mori(S)$
    can be explicitly computed as the minimal nonzero solution $(x,y)$ 
    of the equation
    $$
    y^2+exy-2x^2=1,
    $$
    which is equivalent to $(E')^2=-2$.
    If $e$ is even, i.e., $e=2m$ for $m \ge 1$, 
    then we have $E'=mh-E$.
    If $e$ is odd, then we see that 
    $y$ is odd and $x$ is even ($x=2x'$), and 
    the above equation is reduced to
    the Pell equation
    $X^2-(e^2+8)Y^2=1$, by putting $X:=y+ex'$ and $Y:=x'$.
    Thus the class $x\mathbf h-yE$ of $E'$ in $\Pic S$ is 
    obtained as the minimal solution of this Pell equation
    and computed as in Table~\ref{table:class of E'}.
    \begin{table}[h]
      \begin{center}
	\begin{tabular}{|c|c|c|c|c|c|c|c|c|c|}
	  \hline
	  $e$ & 2 & 3 & 4 & 5 & 6 & 7 & 8 & 9 & $\cdots$ \\
	  \hline
	  $(x,y)$ & $(1,1) $ & $(16,9)$ & $(2,1)$ & $(8,3)$ & $(3,1)$ & $(40,11)$ & $(4,1)$ & $(106000,23001)$ & $\cdots$ \\
	  \hline
	  $d(E')$ & $2$ & $37$ & $4$ & $17$ & $6$ & $83$ & $8$ & $216991$ & $\cdots$ \\
	  \hline
      \end{tabular}
      \end{center}
      \caption{Classes of $E'$}
      \label{table:class of E'}
    \end{table}
  \end{enumerate}
  \label{rmk:class of E'}
\end{rmk}

Secondly, we consider elliptic quartic surfaces.
It follows from Lemma~\ref{lem:mori} that for every integer $e\ge 3$, 
there exists a smooth quartic surface
$S \subset \mathbb P^3$ containing a smooth elliptic curve $F$ of degree
$e$ such that $\Pic S \simeq \mathbb Z \mathbf h \oplus \mathbb Z F$.

\begin{lem}
 \label{lem:mori cone2}
 Let $S$ be a smooth quartic surface, 
 $F$ a smooth elliptic curve of degree $e \ge 4$ on $S$ such that
 $\Pic S=\mathbb Z\mathbf h\oplus \mathbb ZF$,
 $D$ a divisor on $S$.
 Then 
 \begin{enumerate}
  \item $D^2\ne -2$. In particular, there exists no $(-2)$-curve on $S$.
  \item There exists a smooth elliptic curve $F'$ on $S$ such that
  $\mori(S)=\mathbb R_{\ge 0}[F] +\mathbb R_{\ge 0}[F']$.
  Moreover, the class of $F'$ in $\Pic S$ equals
  $e\mathbf h-2F$ if $e$ is odd, and
  $(e/2)\mathbf h-F$ otherwise.
  \item $D$ is nef if and only if $D \ge 0$.
  \item Suppose that $D \ge 0$ and $D\ne 0$.
  Then the following are equivalent:
  (i) every general member $C$ of $|D|$ is a smooth connected curve,
  (ii) $D \not\sim kF$ and $D \not\sim kF'$ for $k \ge 2$,
  (iii) $H^1(S,D)= 0$
 \end{enumerate}
\end{lem}

\begin{proof}
If $D\sim x\mathbf h-yF$ with $x,y \in \mathbb Z$, then
we have $D^2=2x(2x-ye)$ by $h^2=4, h.F=e$ and $F^2=0$.
The Diophantine equation $x(2x-ye)=-1$ on $(x,y)$
has no solutions if $e\ne 1,3$. 
Thus we obtain (1).
Since $F$ spans one of the extremal ray, 
by Kov\'acs's result, the other ray is spanned by the class of 
a smooth elliptic curve $F'$ on $S$.
In fact, by solving the equation $x(2x-ye)=0$,
we see that if $D^2=0$ and $D$ is not spanned by $F$,
then $D$ is spanned by the (primitive) classes given in the lemma.
Thus we obtain (2).
(3) follows from (2),
because $\mori(S)$ is dual to itself.
For (4), suppose that $D \ge 0$ and $D\ne 0$.
Then $|D|$ has no fixed component.
Hence by \cite[Prop.~2.6]{Saint-Donat}, if $D^2>0$, then
every general member $C$ of $|D|$ is a smooth connected curve and $h^1(S,D)=0$,
and otherwise
$D$ is a multiple $kF$ ($k \ge 1$) of a smooth elliptic curve $F$
on $S$ and $h^1(S,D)=k-1$.
\end{proof}

Finally, we consider quartic surfaces without
$(-2)$-curves nor elliptic curves.
Let $S$ be a smooth quartic surface not containing
any $(-2)$-curves nor smooth elliptic curves.
Then $\eff(S)$ coincides with the ample cone $\amp(S)$ of $S$.
In particular, $H^1(S,D)$ vanishes for every effective divisor $D$ on $S$.
If $\rho(S)=2$, then 
$\mori(S)=\mathbb R_{\ge 0}[x_1] +\mathbb R_{\ge 0}[x_2]$
for two non-effective classes $x_i$ ($i=1,2$) with $x_i^2=0$.
We give an example of such a quartic surface.

\begin{ex}
  \label{ex:quartic without P^1's nor elliptics}
  It follows from Lemma~\ref{lem:mori} that
  there exists a smooth quartic surface $S \subset \mathbb P^3$
  and a smooth connected curve $\Gamma$ on $S$
  of degree $6$ and genus $2$
  such that $\Pic(S)=\mathbb Z \mathbf h \oplus \mathbb Z \Gamma$.
  Then on $S$ there exist no divisors $D$
  with $D^2=-2$ and no divisors $D$ with $D^2=0$ and $D \not\sim 0$.
  In fact, let $D\sim x\mathbf h-y\Gamma$ for $x,y \in \mathbb Z$.
  Then we have $D^2=4x^2-12xy+2y^2=(2x-3y)^2-7y^2$.
  We have $D^2\ne -2$, because
  $-2\equiv 5$ is not a quadratic residue modulo $7$.
  If $D\not\sim 0$, then we have $D^2\ne 0$ as well.
  Thus we conclude that there exist no $(-2)$-curves and no 
  smooth elliptic curves on $S$.
\end{ex}

\subsection{Hilbert schemes of $\mathbb P^3$ and $V_4$}
\label{subsec:hilb}

Let $V$ be $\mathbb P^3$ or a smooth quartic $3$-fold 
$V_4 \subset \mathbb P^4$,
$S$ a smooth quartic surface in $\mathbb P^3$
or a smooth hyperplane section of $V_4$.
It is known that if $S$ is general 
(in $|\mathcal O_{\mathbb P^3}(4)|$ or in $|\mathcal O_{V_4}(1)|$),
the Picard group of $S$ is generated by the class $\mathbf h$
of hyperplane sections of $S$ (see e.g.~\cite{Moisezon} for $V=V_4$).
Let $C$ be a smooth connected curve of degree 
$d$ ($=C.\mathbf h$) and genus $g$ in $S$,
not a complete intersection in $S$.
Then there exists a first order deformation $\tilde S$ of $S$
not containing any first order deformation $\tilde C$ of $C$.
Then by Lemma~\ref{lem:flag of k3fano},
$\HF V$ is nonsingular at $(C,S)$ of dimension 
$a^1(C,S)=(-K_V\big{\vert}_S)^2/2+g+1$.
Since $K_{\mathbb P^3}\big{\vert}_S=-4\mathbf h$
and $K_{V_4}\big{\vert}_S=-\mathbf h$,
the number $a^1(C,S)$ equals $g+33$ if $V=\mathbb P^3$, and $g+3$ if $V=V_4$.
Let $W_{C,S} \subset \Hilb^{sc} V$ be 
the $S$-maximal family of curves containing $C$ 
(cf.~Definition~\ref{dfn:S-maximal}).
Then if $d>16$ (resp. $d>4$), 
then we have $H^0(S,N_{S/V}(-C))=0$ and hence
$\dim W_{C,S}=a^1(C,S)$.
In what follows, we assume that $d >16$ if $V=\mathbb P^3$
and $d>4$ if $V=V_4$.

\begin{thm}
  \label{thm:Hilb with rational curve}
  Let $V=\mathbb P^3$ or $V=V_4$, and let $S,C$ be as above.
  Suppose that there exists a smooth rational curve $E$ 
  of degree $e \ge 2$ on $S$ 
  such that $\Pic S=\mathbb Z\mathbf h\oplus \mathbb ZE$.
  Let $E'$ be as in Lemma~\ref{lem:mori cone1} and suppose that
  $D:=C+K_V\big{\vert}_S$ is effective.
  \begin{enumerate}
    \item If $D$ is nef, or $D=E,E'$, then 
    $W_{C,S}$ is a generically smooth component of $\Hilb^{sc} V$.
    \item Suppose that $N_{E/V}$ is globally generated if $V=V_4$.
    If $D.E=-2$ and $D\ne E$, then
    $W_{C,S}$ is a generically non-reduced component 
    of $\Hilb^{sc} V$.
  \end{enumerate}
\end{thm}
\begin{proof}
(1) follows from
Corollary~\ref{cor:main3} and Lemma~\ref{lem:mori cone1}.
We prove (2). For an invertible sheaf $L\sim x\mathbf h+yE$ on $S$,
we have $L.E=0$ if and only if $ex-2y=0$.
This implies that if $L.E=0$,
then the class of $L$ in $\Pic S$ is spanned
by the class $2\mathbf h+eE$ if $e$ is odd, 
and the class $\mathbf h+(e/2)E$ otherwise.
By assumption, there exists an integer $k \ge 1$
such that $D$ is linearly equivalent to the class
$$
\begin{cases}
  k(2\mathbf h+eE)+E & (\mbox{if $e$ is odd}), \\
  k(\mathbf h+(e/2)E)+E & (\mbox{otherwise}).
\end{cases}
$$
Since $(D-E.E)=0$, we have
$D^2=(D-E)^2+E^2=(D-E)^2-2>0$.
We show that $H^1(S,D-3E)=0$.
Suppose that $e$ is even and put $e=2e'$ ($e'\ge 1$).
If $k \ge 2$ or $e'\ge 2$, then 
$D-3E$ is nef and effective, because
$(D-3E.E)=4$ and $(D-3E.E')=k\mathbf h.E'+(ke'-2)E.E'>0$.
If $k=1$ and $e'=1$, then $D-3E=\mathbf h-E=E'$.
Thus we conclude that $H^1(S,D-3E)=0$ by Lemma~\ref{lem:mori cone1}.
Similarly, we can show that $H^1(S,D-3E)=0$ if $e$ is odd.
It follows from Lemma~\ref{lem:non-surjective} that
$\pi_{E/S}(E)$ is not surjective. 
Therefore, by Corollary~\ref{cor:main3},
$W_{C,S}$ is a generically non-reduced component of $\Hilb^{sc} V$.
\end{proof}

In Theorem~\ref{thm:Hilb with rational curve},
if $C\sim a\mathbf h+bE$ ($b \ne 0$ by assumption), then we have
$d=4a+be$ and $g=2a^2+abe-b^2+1$.
Given a class $(a,b)$ of $C$ in $\Pic S$,
one can check whether $D$ is nef or not
by using Lemma~\ref{lem:mori cone1}(3) together with
the method to compute the class of $E'$ shown in
Remark~\ref{rmk:class of E'}(2).

\begin{rmk}
 \label{rmk:Kleppe-Ottem}
 Kleppe and Ottem \cite{Kleppe-Ottem} have also studied
 the deformations of space curves lying on a smooth quartic surface.
 They have considered a smooth quartic surface $S\subset \mathbb P^3$ 
 containing a line $E$, and such that
 $\Pic S \simeq \mathbb Z\mathbf h\oplus \mathbb ZE$,
 and a curve $C \subset S$ of degree $d>16$ and genus $g$,
 not a complete intersection in $S$, satisfying $D:=C-4\mathbf h\ge 0$.
 Then they have proved that
 if $D.E \ge -1$ or $D\sim E$,
 then $W_{C,S}$ becomes a generically smooth component of
 $\Hilb^{sc} \mathbb P^3$, and 
 if $D.E \le -2$ (i.e., $D.E=-2,-3$, or $-4$), $D\ne E$,
 $d \ge 21$ and $g> \min\left\{G(d,5)-1,d^2/10+21\right\}$,
 then $W_{C,S}$ becomes a generically non-reduced component of 
 $\Hilb^{sc} \mathbb P^3$.
 Here $G(d,5)$ denotes the maximum genus of curves of degree $d$ 
 not contained in a surface of degree $4$.
 More recently, we have been informed by Kleppe that if $D.E \le -2$, then
 the assumption on the genus $g$ is almost always satisfied
 (with several exceptions of the classes of $C$ in $\Pic S \simeq \mathbb Z^2$).
 See \cite{Kleppe-Ottem} for the details.
 \end{rmk}

\begin{rmk}
  For a space curve $C \subset \mathbb P^3$, 
  $s(C)$ denotes the minimal degree of surfaces containing $C$.
  An irreducible closed subset $W$ of $\Hilb^{sc} \mathbb P^3$
  is called {\em $s$-maximal},
  if a general member $C$ of $W$ is contained in a surface of 
  degree $s(C)=s$ (i.e. $s(W)=s$), and
  if we have $s(W')>s(W)$
  for any closed irreducible subset $W' \subset \Hilb^{sc} \mathbb P^3$ 
  strictly containing $W$ (cf.~\cite{Kleppe87}).
  If $C \subset \mathbb P^3$ is contained in a surface $S\subset \mathbb P^3$
  of degree $s=s(C)$, then every $s$-maximal subset is 
  a $S$-maximal family, i.e. the image of some irreducible component
  of $\HF^{sc} \mathbb P^3$ passing through $(C,S)$.
  It is known that if a very general curve $C$ of a $4$-maximal family
  $W$ sits on a smooth quartic surface $S$ and $d(C)>16$,
  then the Picard number is at most $2$
  (cf.~\cite[Remark 2.3]{Kleppe-Ottem}).
\end{rmk}

We give infinitely many examples of generically non-reduced 
components of $\Hilb^{sc} V_4$,
which contains Example~\ref{ex:non-reduced}(1)
as a special case ($n=2$).

\begin{ex}
  \label{ex:non-reduced V_4}
  Let $V_4 \subset \mathbb P^4$ be a smooth quartic $3$-fold,
  $E \simeq \mathbb P^1$ a conic on $V_4$ with trivial normal bundle
  $N_{E/V_4} \simeq \mathcal O_{\mathbb P^1}^2$,
  $S$ a smooth hyperplane section of $V_4$ containing $E$
  and such that 
  $\Pic S=\mathbb Z\mathbf h\oplus \mathbb ZE$.
  We consider a complete linear system 
  $\Lambda_n:=|n(\mathbf h+E)|$ on $S$ 
  for an integer $n \ge 2$. 
  Then $\Lambda_n$ is base point free and 
  a general member $C$ of $\Lambda_n$ is a smooth connected curve 
  of degree $6n$ and genus $3n^2+1$.
  Let $W_n$ be the $S$-maximal family $W_{C,S}$ of curves containing $C$.
  Since $D.E=(C-\mathbf h.E)=-2$ and $N_{E/{V_4}}$ is globally generated,
  $W_n$ becomes generically non-reduced components
  of $\Hilb^{sc} V_4$ of dimension $3n^2+4$
  by Theorem~\ref{thm:Hilb with rational curve}.
\end{ex}

\begin{rmk}
It was proved in \cite[Theorem 1.3]{Mukai-Nasu} that
for every smooth cubic $3$-fold $V_3 \subset \mathbb P^4$,
the Hilbert scheme $\Hilb^{sc} V_3$
contains a generically non-reduced component $W$ of dimension $16$.
Then every general member $C$ of $W$ is
a smooth connected curve of degree $8$ and genus $5$, 
and for each $C$,
there exist a smooth hyperplane section $S_3$ of $V_3$
and a line $E$ on $S_3$ satisfying
the linear equivalence
$C \sim 2\mathbf h+2E \sim -K_{V_3}\big{\vert}_{S_3}+2E$.
It was also proved that
if $N_{E/V_3}$ is trivial 
(in other words, $E$ is a {\em good line} on $V_3$), 
then $C$ is primarily obstructed.
We refer to \cite{Nasu4} for a generalization 
to a smooth del Pezzo $3$-fold $V_n$ of degree $n$.
\end{rmk}

We have a similar result for curves lying 
on an elliptic quartic surface.
We have the next theorem as a consequence of
Lemma~\ref{lem:mori cone2}, Lemma~\ref{lem:non-surjective}, 
and Corollary~\ref{cor:main3}.
\begin{thm}
  \label{thm:Hilb with elliptic curve}
  Let $V=\mathbb P^3$ or $V=V_4$, and let $S,C$ be as above.
  Suppose that there exists
  a smooth elliptic curve $F$ of degree $e \ge 4$ on $S$ such that
  $\Pic S=\mathbb Z\mathbf h\oplus \mathbb ZF$.
  Let $F'$ be as in Lemma~\ref{lem:mori cone2} and we assume that
  $D:=C+K_V\big{\vert}_S$ is effective.
  \begin{enumerate}
    \item If $D \not\sim kF$ and $D \not\sim kF'$ for $k \ge 2$, then
    $W_{C,S}$ is a generically smooth component of $\Hilb^{sc} V$.
    \item If $D \sim 2F$ or $D\sim 2F'$, then 
    $W_{C,S}$ is a generically non-reduced component 
    of $\Hilb^{sc} V$.
  \end{enumerate}
\end{thm}

\begin{rmk}
  In Theorem~\ref{thm:Hilb with elliptic curve}, we have
  \begin{enumerate}
    \item If $C\sim a\mathbf h+bF$ 
    ($b \ne 0$ by assumption), then
    $d=4a+be$ and $g=2a^2+abe+1$.
    \item Theorem~\ref{thm:main2} shows that
    if $D\sim mE$ for some integer $m\ge 2$ and an elliptic curve $E$ on $S$,
    then $C$ is obstructed.
    In this case, we have $h^1(S,D)=m-1$.
    However, Theorems~\ref{thm:main1} and \ref{thm:main2}
    are not sufficient 
    for proving that 
    $W_{C,S}$ is an irreducible component
    of $(\Hilb^{sc} \mathbb P^3)_{\red}$ 
    (or $(\Hilb^{sc} V_4)_{\red}$)
    for $m > 2$.
  \end{enumerate}
\end{rmk}

\begin{ex}
  \label{ex:non-reduced P^3}
  Let $V=\mathbb P^3$, $S$ a smooth quartic surface containing
  a smooth elliptic curve $F$.
  Then the complete linear system $\Lambda:=|4\mathbf h+2F|$ on $S$
  is base point free and every general member $C$ of $\Lambda$ 
  is a smooth connected curve on $S$.
  Then $C$ is obstructed by Theorem \ref{thm:main2}, 
  and moreover, $W_{C,S}$ is a generically
  non-reduced component of $\Hilb^{sc} \mathbb P^3$
  by Corollary~\ref{cor:main3}.
\end{ex}

One can compare this example 
with Mumford's example \cite{Mumford} of a generically 
non-reduced component $W$ of $\Hilb^{sc} \mathbb P^3$,
whose general member $C$ is contained in a smooth cubic surface $S$ and
$C\sim 4\mathbf h+2E$ on $S$ for a line $E \subset S$.

\begin{thm}
  \label{thm:Hilb without rational nor elliptic curve}
  Let $V=\mathbb P^3$ or $V=V_4$, and let $S,C$ be as above.
  Suppose that there exist no $(-2)$-curves
  nor elliptic curves on $S$.
  If $D:=C+K_V\big{\vert}_S$ is effective, then
  $W_{C,S}$ is a generically smooth component of $\Hilb^{sc} V$.
\end{thm}

See Example~\ref{ex:quartic without P^1's nor elliptics} 
for an example of a smooth quartic surface 
not containing any $(-2)$-curves nor elliptic curves.

\section*{Acknowledgments}
I should like to thank Prof.~Eiichi Sato for asking me a useful question
concerning the existence of a non-reduced component of
the Hilbert scheme of a smooth Fano $3$-fold of index one.
I am grateful to Prof.~Jan Oddvar Kleppe for very 
helpful comments and a discussion
on non-reduced components of the Hilbert scheme of space curves.
Also thanks to the referee for his/her helpful comments.
This work was partially supported by JSPS Grant-in-Aid (C), No 25400048, 
and JSPS Grant-in-Aid (S), No 25220701.

\end{document}